\documentclass[11pt]{amsart}

\usepackage{graphicx}
\usepackage{color}
\usepackage{url}
\usepackage{hyperref}
\usepackage{mathabx}

\newtheorem{theorem}{Theorem}[section]
\newtheorem{lemma}[theorem]{Lemma}
\newtheorem{prop}[theorem]{Proposition}
\newtheorem{cor}[theorem]{Corollary}
\newtheorem{fact}[theorem]{Fact}

\theoremstyle{definition}
\newtheorem{defi}[theorem]{Definition}
\newtheorem{example}[theorem]{Example}

\theoremstyle{remark}
\newtheorem{remark}[theorem]{Remark}

\numberwithin{equation}{section}

\newcommand{\N}{{\mathbb N}}
\newcommand{\rr}{{\mathbb R}}

\newcommand{\rd}{{\mathbb R^d}}

\newcommand{\E}{{\mathbb E}}

\newcommand{\DOA}{\operatorname{DOA}}

\newcommand{\RV}{\operatorname{RV}}

\newcommand{\eqd}{\stackrel{d}{=}}

\def\cal{\mathcal}
\date{November 17, 2014}
\keywords{regular variation; regular variation on cones; max-stable; implicit max-stable; implicit order statistics}
\begin{document}

\title{Implicit Extremes and Implicit Max--Stable Laws}

\author{Hans-Peter Scheffler$^\dagger$}\thanks{$\dagger$ University of Siegen, Germany, {\tt scheffler@mathematik.uni-siegen.de}}

\author{Stilian Stoev$^\ddagger$}\thanks{$\ddagger$ University of Michigan, Ann Arbor, USA, {\tt sstoev@umich.edu}}

\maketitle

\begin{abstract} Let $X_1,\cdots,X_n$ be iid random vectors and $f\ge 0$ be a non--negative function. 
Let also $k(n) = {\rm Argmax}_{i=1,\cdots,n} f(X_i)$. We are interested in the distribution of 
$X_{k(n)}$ and their limit theorems. In other words, what is the distribution the random vector where
a function of its components is extreme. This question is motivated by 
a kind of inverse problem where one wants to determine the extremal behavior of $X$ when 
only explicitly observing $f(X)$. We shall refer to such types of results as to {\em implicit extremes}.
It turns out that, as in the usual case of explicit extremes, all limit {\em implicit extreme value} laws are {\em implicit 
max--stable.} We characterize the regularly varying implicit max--stable laws in terms of their spectral and stochastic
representations.  We also establish the asymptotic behavior of {\em implicit order statistics} relative to a
given homogeneous loss and conclude with several examples drawing connections to prior work involving regular 
variation on general cones.  
\end{abstract}

\section{Introduction}
On January 21 in 1959, Ohio experienced an extreme flood, which was the most destructive such event since 1913
claiming 16 lives and \$100 million in damages.  
The root cause of this event was not entirely due to extreme precipitation. It was essentially due 
to {\em rain on frozen ground}, i.e., cold ground--freezing conditions combined with a rare case of moderately 
intensive rainfall due to a warm front 
\cite{floods:1959}.  In hydrology, it is well understood that floods are not simply caused by 
extreme precipitation but, in fact, involve a complex combination of factors including ground saturation, snow--melt, 
precipitation intensity, and duration.  In such and many other applications {\em extreme loss} events
are caused by unusual combination of factors, the marginal values of which may or may not be {\em extreme} but 
their coordinated effect is extreme.  
Such type of phenomena motivated us to focus on {\em extreme loss events} rather than extreme values and  
develop theory that helps understand and model the joint behavior of the factors leading to extreme losses.  

More precisely, let $f: \mathbb R^d \to [0,\infty)$ be a non--negative function modeling the loss $f(x)$ associated with 
the values $x=(x_i)_{i=1}^d$ of $d$ factors.  Let also $X$ be a random vector in $\mathbb R^d$, 
modeling the joint behavior of these $d$ factors. Assuming that $X_1,\ldots,X_n$ are independent copies (measurements)
of $X$, we are interested in the behavior of $X_{k(n)}$ leading to maximal loss.  Namely, let
$$
k(n) := {\rm Argmax}_{k=1,\ldots,n} f(X_k),
$$
where in the case of ties $k(n)$ is taken as the smallest index yielding the maximum.

In this paper, our main goal is to establish the asymptotic behavior of $X_{k(n)}$ under appropriate normalization.
Operationally, $X_{k(n)}$ may be viewed as the {\em $f$--implicit maximum} of the $X_k$'s, i.e., the observation 
leading to maximal $f$--loss. As illustrated,  the events leading to extreme losses $f(X)$ are of utmost importance 
in practice.  Thus, given a loss functional of interest, the limit distribution of $X_{k(n)}$, as $n\to\infty$ provides 
a natural fundamental model for the joint dependence structure of the factors leading to such extremes.

Here, we focus on the case of {\em homogeneous} losses (see \eqref{eq4}, below).  In Theorem \ref{thm1},
under the assumption that $X$ is {\em regularly varying} on the cone $\mathbb R^d\setminus\{f=0\}$, 
we show that
\begin{equation}\label{e:intro-X-to-Y}
\frac{1}{a_n} X_{k(n)}  \Longrightarrow Y,\ \ \mbox{ as }n\to\infty,
\end{equation}
for some normalizing sequence $a(n)>0$, where `$\Rightarrow$ denotes convergence in distribution.

The limit laws arising in \eqref{e:intro-X-to-Y} will be referred to as {\em $f$-implicit extreme value distributions}. 
As anticipated from the classic theory of
(explicit) multivariate extremes, the limits in \eqref{e:intro-X-to-Y} have certain stability property with respect to
the operation of {\em implicit maxima}. Indeed, if $Y_k,\ k=1,\ldots, n$ are independent copies of $Y$, then
it turns out that, for all $n$, exists $a(n)>0$ such that 
\begin{equation}\label{e:intro-imp-max-stab}
 {\rm Argmax}_{Y_k,\ k=1,\ldots,n} f(Y_k) \stackrel{d}{=} a(n) Y. 
\end{equation}
Random vectors satisfying \eqref{e:intro-imp-max-stab} will be referred to as {\em $f$-implicit max--stable}. Our 
first result (Theorem \ref{thm1}) shows that all implicit extreme value laws are in fact implicit max--stable. The converse
is also true.  In Theorem \ref{thm:DOA_char}, 
we characterize the implicit max--domain of attraction of all $f$-implicit max--stable laws 
associated with positive and continuous homogeneous loss functions $f$.  It turns out that these laws are precisely 
the regularly varying distributions on the cone $\mathbb R^d\setminus\{f=0\}$.  This result shows that the generalized 
notion of regular variation on cones is a natural technical and conceptual approach to implicit extremes. 
The notion of regular variation on general cones originates from the works on hidden regular variation 
of Resnick and Maulik \cite{resnick:maulik:2004}. It is briefly defined and reviewed in Section \ref{sec:RV} 
below from the perspective of generalized polar coordinates.  More details and further applications or regular 
variation on cones can be found in the recent work of \cite{lindskog:resnick:roy:2013}.

In Section \ref{sec:examples}, we discuss several examples that unveil connections to prior work by 
Ledford and Tawn \cite{ledford:tawn:1996}, Draisma {\em et al} \cite{draisma:drees:ferreira:dehaan:2004}, and
de Haan and Zhou \cite{dehaan:zhou:2011}.  The recent work of Dombry and Ribatet \cite{dombry:ribatet:2014}
on $\ell$--Pareto processes involves very similar ideas to ours. It focuses on the limit behavior of a process $X$ conditionally on the event that a certain loss 
functional $\ell(X)$ is extreme.  In this sense, Dombry and Ribatet study implicit {\em exceedances} whereas
we study implicit {\em maxima}. Technically, the two approaches: {\em implicit extremes} and {\em implicit exceedances}
lead to different limits and contexts of application but they both have important virtues. Conceptually, implicit 
extremes correspond to the study of (implicit) maxima, while implicit exceedances to (implicit) peaks-over-threshold 
as defined by a suitable loss functional.

\medskip
\noindent{\em The paper is structured as follows.} In Section \ref{sec:prelim}, we start with 
the problem formulation and give a key technical lemma.  In Section \ref{sec:RV}, we review
regular variation on cones in $\mathbb R^d$ using generalized polar coordinates.  We provide
a disintegration formula for the measure of regular variation via its spectral measure on a
(generalized) unit sphere. In Section \ref{sec:imp_ext}, we establish limit theorems for implicit extremes 
under regular variation condition. We also give stochastic representations of the limit laws, where the
disintegration formula plays an important role.  The implicit maximum domain of attraction is characterized
in Section \ref{sec:DA}.  In Section \ref{sec:o-stat}, we define implicit order statistics relative to a given loss and
establish their asymptotic behavior.  We end with several examples and discuss related work 
in Section \ref{sec:examples}.  Some technical proofs and auxiliary results are given in the Appendix.

\section{Preliminaries} \label{sec:prelim}

Let $X_i=(X_i^{(1)},\dots,X_i^{(d)})$ be iid $\rr^d$-valued random vectors. Moreover, let $f : \rr^d \to\rr_+$ be 
measurable. For $n\geq 1$ define 
\begin{equation}
 \label{e:k(n)-def}
  k(n)=\operatorname{argmax}\bigl\{f(X_1),\dots,f(X_n)\bigr\} 
\end{equation}
so that
\begin{equation}
 \label{e:X-k(n)-def}
  f(X_{k(n)})=\max\bigl\{f(X_1),\dots,f(X_n)\bigr\}.
\end{equation}
In the case of ties, $k(n)$ is taken as the smallest index for which the maximum is attained.

We are interested in the distribution of $X_{k(n)}$ and their limit theorems. In other words, what is the joint distribution 
of the components of the random vector where a function of its components is extreme. Broadly speaking, this is 
motivated by a kind of inverse problem where one wants to determine the extremal behavior of $X$ when 
only explicitly observing $f(X)$. This is why we shall refer to such types of results as to {\em implicit extremes}.

\begin{lemma}\label{lem1} Let $G(y):=P(f(X)\le y)$ be the distribution function of $f(X)$,
where $X\eqd X_1$. Then, for all measurable $A\subset\rr^d$ we have
\begin{equation}\label{eq1-new}
 n\int_A G(f(x)-)^{n-1}\,P_X(dx) \le   P\bigl\{X_{k(n)}\in A\bigr\}\le n\int_A G(f(x))^{n-1}\,P_X(dx).
\end{equation}
In particular, if $G$ is continuous, then
\begin{equation}\label{eq2-new}
P_{X_{k(n)}}(dx)=n G(f(x))^{n-1}\,P_X(dx) \equiv n P\bigl\{f(X)\leq f(x)\bigr\}^{n-1}\,P_X(dx).
\end{equation}
\end{lemma}

\begin{proof} We have that
\begin{eqnarray}\label{e:lem1-1}
& & P\bigl\{X_{k(n)}\in A\bigr\} =\sum_{i=1}^n P\bigl\{X_i\in A, k(n)=i\bigr\} \nonumber\\ 
 & &\quad =\sum_{i=1}^n P\bigl\{X_i\in A, f(X_j)< f(X_i),\ 1\le j< i\ \text{ and } f(X_j)\le f(X_i), \ i<j\le n\bigr\}.
\end{eqnarray}  
Each term in the above sum is bounded above by 
$$
P\bigl\{X_i\in A, f(X_j)\leq f(X_i)\ \text{for all $j\neq i$}\bigr\}.
$$
Therefore, by using the fact that the $X_i$s are iid, we obtain 
\begin{eqnarray*}
P\bigl\{X_{k(n)}\in A\bigr\} &\leq &
 n \int_A P\bigl\{f(X_j)\leq f(x)\ \text{for all $j\neq i$} | X_i=x\bigr\}\,  P_X(dx) \\
 &=&  n \int_A P\bigl\{f(X)\leq f(x)\bigr\}^{n-1}\,P_X(dx) \\
 &=& \int_A G(f(x))^{n-1}\,P_X(dx).
\end{eqnarray*}
This yields the upper bound in \eqref{eq1-new}. Similarly, each term of the sum in the 
right-hand side  of \eqref{e:lem1-1} is bounded below by
\begin{eqnarray*}
 P\bigl\{X_i\in A, f(X_j)< f(X_i)\ \text{for all $j\neq i$}\bigr\} 
 &=&  \int_A P\bigl\{f(X)< f(x)\bigr\}^{n-1}\,P_X(dx) \\
 &=& \int_A G(f(x) - )^{n-1}\,P_X(dx),
\end{eqnarray*}
which completes the proof of \eqref{eq1-new}.
\end{proof}

\section{Implicit extreme value laws}\label{sec1}

In this section, we establish limit theorems for $X_{k(n)}$ in \eqref{e:X-k(n)-def}. 
The emerging limits will be referred to as implicit extreme value distributions.
To this end, we need to impose some assumptions on $X$ and $f$. Since we are concerned with 
multivariate extreme value theory on $\rr^d$, it is natural to work in the context of multivariate regular variation.  
We shall need, however, a slight extension, which considers this notion over general {\it sub-cones} of 
$\overline\rr^d$. 

\subsection{Regular variation on cones}\label{sec:RV}  This exposition is motivated by the fundamental concept of {\em hidden regular variation}
pioneered by Resnick and Maulik \cite{resnick:maulik:2004} (see also p.\ 324 in \cite{resnick:2007}). The recent work of 
\cite{lindskog:resnick:roy:2013} develops abstract and far-reaching theory in the context of metric spaces. The following presentation 
is tailored to our needs.

Let $\overline \rr$ denote the extended Real line $[-\infty,\infty]$. The topology in $\overline\rr$ is generated by the 
usual class of open sets in $\rr$ along with the open neighborhoods of $\pm\infty$ of the type $(a,\infty]$ and 
$[-\infty,a),\ a\in\rr$. Thus $\overline\rr$ becomes compact.

Let also $\overline\rr^d$ be the Cartesian $d$-power of the extended Real line, equipped with the product topology. The 
space $\overline \rr^d$ is compact by Tichonoff's theorem. It is also separable and complete with respect to the metric
\begin{equation}\label{e:rho}
 \rho(x,y) := \sum_{i=1}^d r(x_i,y_i),\ \ x = (x_i)_{i=1}^d, y = (y_i)_{i=1}^d\in \overline{\rr}^d,
\end{equation}
where $r(x,y):= |{\rm atan}(x) - {\rm atan}(y)|$, where ${\rm atan}(\pm \infty) := \pm \pi/2$.
In fact, $\overline \rr^d$ is homeomorphic to the compact interval 
$[-\pi/2,\pi/2]^d$ equipped with the usual topology, where the map $x\mapsto {\rm atan}(x)$ taken 
coordinate-wise is one homeomorphism, for example. 
 
The classic notion of multivariate regular variation involves  the `punctured space' $\overline \rr^d \setminus\{0\}$. In our context, it is 
convenient to remove an entire cone rather than just the origin. Recall that $D\subset \overline\rr^d$ is said to be a (positive) 
cone, if  $\lambda D \subset D$, for all $\lambda >0$.

Let $D$ be a {\it closed positive cone} and consider the punctured space $\overline \rr^d_D:= \overline \rr^d\setminus D$, equipped with the 
relative topology.  As expected, we have the following characterization of compacts.

\begin{fact}\label{fact:compact} A set $F\subset \overline \rr^d_D$ is compact if and only if it is closed and
bounded away from $D$, that is,  $F \subset \overline \rr^d \setminus U$, where $U \supset D$ is 
an open neighborhood  of $D$ in $\overline \rr^d$. 
\end{fact}

\noindent The proof is given in the appendix. 
We equip $\overline\rr^d_D$ with the Borel $\sigma$-algebra generated by all open sets. We shall consider Radon measures on 
$\overline\rr^d_D$, i.e.\ those that are finite on all compacts. Since $\overline \rr^d_D$ can be represented as 
a countable union of compacts, the Radon measures are $\sigma$-finite. Recall that the Radon measures  $\nu_n$ converge
vaguely to another measure $\nu$, written
$$
\nu_n\stackrel{v}{\to}v,\ \ \mbox{ as }n\to\infty,
$$
if and only if $\int h d\nu_n \to \int h d\nu,\ n\to\infty$, for all continuous $h:\overline \rr^d_D\to \rr$
that vanish outside some compact set in $\overline \rr^d_D$. The limit $\nu$ is necessarily Radon.

\begin{defi}\label{d:RV-D} Let $D$ be a closed cone in $\overline\rr^d$. A random vector $X$ in $\rr^d$ 
is said to be regularly varying on $\rr^d_D:= \rr^d\setminus D$ with exponent $\alpha>0$, if there exists
a non-trivial Radon measure $\nu$ on $\overline\rr^d_D$ supported on $\rr^d_D$ 
and a regularly varying sequence $a_n>0$ with exponent $1/\alpha$, such that 
\begin{equation}\label{e:d:RV-D}
 n P(a_n^{-1}X \in \cdot) \stackrel{v}{\longrightarrow} \nu,\ \ \mbox{ as }n\to\infty.
\end{equation}
In this case, we write $X \in RV_\alpha(\{a_n\},D,\nu)$ or sometimes simply $X\in RV_\alpha(D,\nu)$.
\end{defi}

%
%
%
%
%
Observe that the vague convergence in \eqref{e:d:RV-D} involves measures defined on 
$\overline\rr^d_D$ that vanish on the set of infinite points $\overline \rr^d_D \setminus \rr^d.$
In the case when the {\em exceptional cone} is $D = \{0\}$, one recovers the usual notion 
of multivariate regular variation.

\begin{remark} The above definition is closely related and in fact inspired
by the fundamental concept of {\em hidden regular variation} of Resnick and Maulik 
(see e.g.\ p.\ 324 in \cite{resnick:2007}). Our definition, however, does not involve multiple cones 
and it does not require, in particular, that $X$ be
multivariate regularly varying on $\rr^d_{\{0\}}$. In this sense, regular variation on cones is both more basic and less restrictive
than hidden regular variation. For a general treatment and several equivalents to the above Definition \ref{d:RV-D}, 
see \cite{lindskog:resnick:roy:2013} and also Proposition \ref{p:RV-polar} below.
\end{remark}

\medskip
The limit $\nu$ in \eqref{e:d:RV-D} has the scaling property  
\begin{equation}\label{e:nu-scaling}
\nu(\lambda \cdot) = \lambda^{-\alpha} \nu(\cdot ),\ \ \mbox{ for all $\lambda>0$ }
\end{equation}
(see e.g.\ Theorem 3.1 in \cite{lindskog:resnick:roy:2013}).
As in the classical case, \eqref{e:nu-scaling} yields a {\it disintegration formula} 
for $\nu$ involving radial and angular (or spectral) components. Special care needs to be taken, however, in defining the 
unit sphere. One may take as the unit sphere the boundary of any star-shaped domain containing $D$ in its interior 
(in $\overline \rr^d$). In practice, however, it is easier to derive it from suitable generalized polar coordinates 
as follows.

\begin{defi}[polar coordinates in $\rr^d\setminus D$] \label{d:polar-coors}
Let $D$ be a closed cone in $\overline\rr^d$. Let also $\tau:\overline\rr^d \to [0,\infty]$ be a continuous 
function such that $\{\tau = 0\} = D$ and $\tau(x)<\infty$ for all $x\in \rr^d$.  We shall assume also that $\tau$ is
$1$-homogeneous, that is,  $\tau (\lambda x) = \lambda \tau (x)$ for all $\lambda> 0$ and $x\in \overline \rr^d$. 
For $x\in \overline\rr^d\setminus (D\cup\{\tau = \infty\}) \supset \rr^d\setminus D$,
its {\em polar coordinates} are defined as 
\begin{equation}\label{e:polar-coors}
 (\tau,\theta):= (\tau(x), x/\tau(x)),
 \end{equation}
where $\tau$ is referred to as the radial and $\theta$ is the angular component of $x=\tau \theta$.
\end{defi}

Now, fix some polar coordinates as in \eqref{e:polar-coors}. The corresponding {\it unit sphere} is
$$
 \overline S:= \{x\in \overline \rr^d_D\, :\, \tau(x) = 1\}.
$$
Since $\tau$ is continuous in $\overline \rr^d_D$, the set $\overline S$ is closed in $\overline \rr^d_D$. It is also
bounded away from $D \equiv\{\tau = 0\}$. Hence, the unit sphere $\overline S$ 
is compact (in $\overline \rr^d_D$). Consider the relative topology and corresponding Borel $\sigma$-algebra
induced on $\overline S$. It can be shown that the map 
\begin{equation}\label{e:T-def}
 T:\overline\rr^d_D\setminus\{\tau = \infty\} \to (0,\infty)\times
 \overline S,
\end{equation}
 defined as $T(x):= (\tau(x),\theta(x))$, is a homeomorphism of topological spaces. 
 That is, $T$ is one-to-one and onto, and both $T$ and its inverse $T^{-1}$ are continuous and hence measurable. 
 The restriction of the map $T: \rr^d\setminus D \to (0,\infty)\times S$, where 
 $$
  S := \overline S \cap \rr^d \equiv \{ \tau =1 \} \cap \rr^d 
 $$ 
 is also a homeomorphism. The difference between $\overline S$ and $S$ is that the former may (and typically will) 
 contain infinite points in $\overline \rr^d\setminus \rr^d.$  Since the measures $\nu$ involved in \eqref{e:d:RV-D} 
 are supported on $\rr^d$, however, we shall work with the uncompactified unit sphere $S$ that contains only 
 points from $\rr^d$.

Any measure $\nu$ on $\rr^d\setminus D$ naturally induces a measure $\widetilde \nu := \nu \circ T^{-1}$ 
on $(0,\infty)\times S$, where $(0,\infty)\times S$ is equipped with the product $\sigma$-algebra. 
The disintegration formula for $\nu$ in \eqref{e:fact:disintegration} below is a consequence of the
scaling property \eqref{e:nu-scaling} and a change of variables.  To gain intuition consider the `cylinder sets' 
$$
  A_{r,B} = \{ x\, : \, \tau(x) >r, \,  \theta(x) \in B\} = T^{-1}( (r,\infty)\times B),\  \ r>0,\ B\subset S
$$ 
and observe that by the scaling property
\begin{equation}\label{e:nu-A.r.B}
\nu(A_{r,B}) = r^{-\alpha} \nu(A_{1,B}) = \widetilde \nu( (r,\infty)\times B ).
\end{equation}
This suggests defining the measure $\sigma_S$ on $S$ as follows
\begin{equation}\label{e:sigma}
\sigma_S(B) := \nu(A_{1,B}) \equiv \widetilde \nu( (1,\infty)\times B),\ \  B\subset S.
\end{equation}
Note that $\sigma_S (S) = \nu\{\tau >1\}<\infty$, since $\{\tau >1\}$ is bounded away from $D$. 
By writing the term $r^{-\alpha}$  in \eqref{e:nu-A.r.B} as $\int_r^\infty \alpha \tau^{-\alpha-1}d\tau$, we obtain
the following result.
  
\begin{fact}\label{fact:disintegration} Let $\nu$ be a Radon measure on $\overline\rr^d_D$, supported on $\rr^d$,
which satisfies \eqref{e:nu-scaling}, for some $\alpha>0$. Let also $(\tau,\theta)$ be polar coordinates on
$\rr^d_D$ as in \eqref{e:polar-coors}.
Then, there exists (unique) finite measure $\sigma_S$ on $S:= \{\tau =1\}\cap \rr^d$, such that
\begin{equation}\label{e:fact:disintegration} 
\nu(A)  =\int_{S}\int_0^\infty  1_{A}(\tau \theta) \frac{\alpha d\tau}{\tau^{\alpha+1}} \sigma_{S}(d\theta),
\end{equation}
for all measurable $A\subset \rr^d_D$. The measure $\sigma_S$ is uniquely identified by \eqref{e:sigma}.
\end{fact}

The measure $\sigma_{S}$ in \eqref{e:fact:disintegration} will be referred to as a
{\em spectral measure} of $\nu$ and will be used in the sequel to conveniently represent implicit extreme value laws.
Depending on the cone $D$, the `right' choice of polar coordinates and resulting unit `sphere' may be somewhat 
counter-intuitive in applications as the following example shows. See also Example 3.1 in \cite{lindskog:resnick:roy:2013}.

\begin{example}[Pareto and Dirichlet] \label{ex:Pareto-Dirichlet} Let $X = (U_i^{-1/\alpha_i})_{i=1}^d, $ where $U_i\stackrel{iid}{\sim}$Uniform$(0,1)$ and $\alpha_i>0,\ i=1,\cdots,d$. That is,
the components of $X$ are independent standard $\alpha_i$-Pareto. It is well known that $X$ is regularly varying in the 
usual sense, where the measure of regular variation $\nu$ 
concentrates on the coordinate axes corresponding to heaviest tail(s). More precisely, 
$X\in RV_\alpha(\{0\},\nu)$, where $\alpha:= \min_{i=1,\cdots,d} \alpha_i$, and in this case
$$
 \nu( \bigtimes_{i=1}^d (x_i,\infty) ) = \sum_{i=1}^d \delta_{\alpha}(\alpha_i) \frac{1}{x_i^{\alpha+1}},\ \ \mbox{ for all }x_i>0.
$$
However, if one excises the axes as well as the origin, the random vector $X$ becomes regularly varying with non--trivial measure $\nu$
supported on the entire positive orthant for all possible choices of positive exponents $\alpha_i$. 

Indeed, focus on the strictly positive orthant $\rr^d_D := (0,\infty)^d$. Since
$$
P( X^{(i)} >x_i,\ i=1,\cdots,d ) = \prod_{i=1}^d x_i^{-\alpha_i} =: \nu(  \bigtimes_{i=1}^d (x_i,\infty) ),\ \ \mbox{ for all } x_i\ge 1,
$$
it is easy to see that $X \in RV_\alpha(\{n^{1/\alpha}\},D,\nu)$, where 
\begin{equation}\label{e:Pareto-alpha-nu}
 \alpha = \sum_{i=1}^d\alpha_i, \ \ \mbox{ and }\ \ \nu(dx_1\cdots d x_d) = \prod_{i=1}^d \frac{\alpha_i}{x_i^{\alpha_i+1}} \times  dx_1\cdots dx_d. 
\end{equation}
We shall now determine the spectral measure of $\nu$ in suitable polar coordinates. Let 
\begin{equation}\label{e:tau-Pareto}
 \tau(x) = \left( \sum_{i=1}^d \frac{1}{(x_i)_+} \right)^{-1} \equiv \| 1/x_+ \|_{\ell_1}^{-1}\ \ \mbox{ and }\ \ \theta(x) = x/\tau(x).
\end{equation}
Observe that $\tau:\overline\rr^d \to [0,\infty]$ is $1$--homogeneous, continuous, $\tau(x)<\infty,\ x\in\rr^d$,
and $\{\tau>0\}= (0,\infty]^d$. Therefore, $(\tau,\theta)$ are
valid polar coordinates in $(0,\infty)^d\subset \overline \rr^d\setminus\{D\cup \{\tau = \infty\})$. 
Note also that the unit sphere $S = \{\tau = 1\} \cap \rr^d$ can be parameterized as follows:
$$
 S = {\Big\{} (1/u_i)_{i=1}^d\, :\, u_i \in (0,1),\ \sum_{i=1}^d u_i = 1 {\Big\}}.
$$
That is, $S$ is the image of the open unit simplex with respect to the coordinate-wise inversion operation ${\cal I}(u_1,\cdots,u_d) = (1/u_1\cdots 1/u_d)$.
With this parameterization, we have that
$x_i = \tau/u_i,\ i=1,\cdots,d$ and a standard computation of Jacobians yields
$$
  dx_1 \cdots dx_d = \tau ^{d-1} \prod_{i=1}^d u_i^{-2} \times d\tau du_1\cdots d u_{d-1},
$$
where the free variables are $\tau \in (0,\infty)$ and $u_i,\ i=1,\cdots,d-1$ with $u_i> 0, \sum_{i=1}^{d-1} u_i < 1$. For notational convenience
we let  $u_n:= (1-\sum_{i=1}^{d-1}u_i)$.

Now, for the measure $\nu$ in \eqref{e:Pareto-alpha-nu}, we obtain
\begin{eqnarray}\label{e:Dirichlet-spectrum}
 \nu(d\tau du_1\cdots d u_{d-1}) &=& \prod_{i=1}^d \frac{\alpha_i u_i^{\alpha_i+1} }{\tau^{\alpha_i+1}} \times \tau ^{d-1} \prod_{i=1}^d u_i^{-2} \times d\tau du_1\cdots d u_{d-1}\nonumber\\
 & =& \frac{\alpha d\tau}{\tau^{\alpha +1}}\times \frac{\prod_{i=1}^d \alpha_i}{\alpha}  u_1^{\alpha_1-1} \cdots u_{d}^{\alpha_{d}-1} du_1\cdots d u_{d-1} 
    =: \frac{\alpha d\tau}{\tau^{\alpha+1}} \times \sigma_S(d\theta).
  \end{eqnarray}
This calculation shows an intriguing fact that the spectral measure $\sigma_S$ in \eqref{e:Dirichlet-spectrum} is up to a constant the lift of a Dirichlet distribution 
on the unit simplex. That is, with ${\cal I}(x)= 1/x$, we have that
$$
\sigma_S(B) = c_{\{\alpha_i\}} P( {\cal I}(\xi)  \in B ),\ \ \mbox{ where }\xi\sim {\rm Dirichlet}(\alpha_1,\cdots,\alpha_d),
$$
and where $c_{\{\alpha_i\}} = (\prod_{i=1}^d \alpha_i \Gamma(\alpha_i))/(\alpha \Gamma(\alpha)).$

This result can be used to efficiently simulate from implicit max-stable laws  and in fact to characterize all such laws that have spectral measures absolutely continuous with 
respect to $\sigma_S$ (see Example \ref{ex:Pareto-Dirichlet-continued} and Proposition \ref{p:Pareto-Dirichlet-tilting}, below). 
\end{example}

  \begin{remark} Other choices of polar coordinates are possible with the caveat that the unit `sphere' needs to be bounded away from $D$.
  The typical choice of a unit sphere $S = \{ \|x\| =1\} \setminus D,$ where $\|\cdot\|$ is some norm in $\rr^d$ would
  not have worked well in the previous example. Indeed, it could provide a disintegration formula of the type \eqref{e:fact:disintegration}, but the resulting spectral measure will
  be infinite. This is because the set $S$ is not bounded away from $D$. 
  \end{remark} 
  
  \begin{remark} For most cones $D$ it is not possible to extend the homogeneous 
  polar coordinates as a homeomorphism to the {\em entire} space $\overline \rr^d\setminus D$ (including 
  all points at infinity). Indeed, consider Example \ref{ex:Pareto-Dirichlet}, where $\overline \rr^d_D = (0,\infty]^d$ 
  and  $(\tau,\theta)$ are as in \eqref{e:tau-Pareto}. Then, $(\tau,\theta): (0,\infty]^d\setminus\{(\infty,\cdots,\infty)\} \to 
    (0,\infty)\times \overline S$ is a  homeomorphism.  Since $\{\infty\} \times \overline S$ is 
    not homeomorphic to the single point $(\infty,\cdots,\infty)$, however, the polar coordinates do not extend 
     to $\overline \rr_D^d$. 
     
     In the classic case of $\overline \rr^d_{\{0\}}$, the coordinates $\tau(x):= \|x\|$, $\theta(x):=x/\|x\|$,
     extend by continuity to $\overline\rr^d_{\{0\}}$, where $\|\cdot\|$ is an arbitrary norm in $\rr^d$. 
     This is perhaps the only case when $(\tau,\theta): \overline\rr^d_{\{0\}} \to (0,\infty] \times \overline S$ 
     is a homeomorphism. 
  \end{remark}
 
 We give next a version of the well-known characterization of regular variation on $\rr^d_D$ in terms 
 of generalized polar coordinates. The proof is given in the Appendix.
 
\begin{prop}\label{p:RV-polar} Let $(\tau,\theta)$ be polar coordinates in $\rr^d_D$ as 
in Definition \ref{d:polar-coors}. Then $X\in RV_\alpha(\{a_n\},D,\nu)$ if and only if, for some 
$C>0$ and all $x>0$
\begin{equation}\label{e:p-RV-polar}
  n P( a_n^{-1} \tau(X) >x) \mathop{\longrightarrow}_{n\to\infty} C x^{-\alpha} \quad{ and }\quad\ P(\theta(X) \in \cdot\, \vert\, \tau(X)>u) 
  \mathop{\stackrel{w}{\longrightarrow}}_{u\to\infty} \sigma_0(\cdot),
\end{equation}
where $\sigma_0(\cdot)$ is a probability measure on the unit sphere $S$. In this case, the spectral measure 
$\sigma_S$ of $\nu$ and $\sigma_0$ are related as follows $\sigma_S = C \sigma_0$, where $C = \nu(\{\tau>1\})$.  
 \end{prop}

We give next an extension of the standard Breiman-type lemma, which provides a useful way
of constructing regularly varying distributions on cones.

\begin{lemma}[Breiman in polar coordinates]\label{l:Breiman}  Let $X:=ZV$, where $Z$ and $V$ be independent and such that 
$Z$ is a positive random variable and $V$ takes values in the cone $\rr^d\setminus D$. 
Then, the conditions
 $$
  P(Z>x) \sim x^{-\alpha},\ x\to\infty\ \quad \mbox{ and }\ \quad E(\tau^\alpha (V)) <\infty,
$$
imply that for all $x>0$
$$
nP(n^{-1/\alpha} \tau(X) >x) \mathop{\longrightarrow}_{n\to\infty} E(\tau^\alpha(V)) x^{-\alpha}
\ \ \mbox{ and }\ \  P(\theta(X) \in \cdot\, \vert\, \tau(X)>u) 
\mathop{\stackrel{TV}{\longrightarrow}}_{u\to\infty} \sigma_V(\cdot),
$$
where the last convergence is in the sense of total variation norm and
$$
\sigma_V(B):= \frac{1}{E \tau^\alpha(V)} \int_{\rr^d} 1_B(\theta(v)) \tau^\alpha(v) P_V(dv).
$$
In particular, $X \in RV_\alpha(\{n^{1/\alpha}\},D,\nu)$, where the spectral measure of 
$\nu$ is $\sigma_S(\cdot) = E (\tau^\alpha(V))\, \sigma_V(\cdot)$.
\end{lemma}
\begin{proof} By the extension of Breiman's lemma given in Lemma 2.3 (2) of \cite{davis:mikosch:2008}, we have 
\begin{equation}\label{e:example-Breiman-1}
 P(\tau(X)>u) = P( Z \tau(V) > u) \sim u^{-\alpha}E(\tau^\alpha (V)),\mbox{ as }u\to\infty.
\end{equation}
Now, for all measurable $B\subset S\equiv \{\tau = 1\}\cap \rr^d$, by homogeneity and independence 
\begin{eqnarray}\label{e:example-Breiman-2}
& &  P(\theta(X) \in B\, \vert\, \tau(X)>u) = \frac{P( \theta(V) \in B,\, Z\tau(V)>u)}{P(\tau(X)>u)}   \nonumber\\
 & &\quad \quad  =  \int_{\rr^d} 1_{B}(\theta(v)) \frac{P(Z> u/\tau(v)) }{ P(\tau(X)>u)}  P_V(dv) =: \int_{\rr^d} 1_{B}(\theta(v)) h_u(v)  P_V(dv),
\end{eqnarray}
where $P_V$ stands for the law of $V$.

By setting $B = S$, we see that $h_u$ are probability densities. Further, by \eqref{e:example-Breiman-1} and 
since $u^{\alpha}P(Z>u/c) \to c^\alpha,$ as $u\to\infty$, for all $c>0$, we get that 
$$
h_u(v) \longrightarrow h(v):= \frac{\tau^\alpha(v)}{ E \tau^\alpha(V)},\ \mbox{ as }u\to\infty,
$$ 
where the convergence is valid for all $v$ since $h_u(v)\equiv 0$, by convention when $\tau(v) = 0$.  Note that
$h$ is also a probability density with respect to $P_V$.  Thus, by the Scheffe-type Lemma \ref{l:Scheffe}, we get that, as $u\to\infty$,
$$
P(\theta(X) \in \cdot  \, \vert\, \tau(X)>u) \stackrel{{\rm TV}}{\longrightarrow} \sigma_V(\cdot) := \int_{V} 1_{(\cdot)} (\theta(v)) \frac{\tau^\alpha(v)}{ E \tau^\alpha(V)} P_V(dv).
$$
This along with \eqref{e:example-Breiman-1} thanks to Proposition \ref{p:RV-polar}, implies that 
$X = ZV \in RV_\alpha(\{n^{1/\alpha}\},D,\nu)$, where the spectral measure 
of $\nu$ is 
$
  \sigma_S(\cdot) = E(\tau^\alpha(V))\, \sigma_V(\cdot).
$
\end{proof}

\subsection{Limit theorems for implicit extremes} \label{sec:imp_ext}

We start by listing the assumptions on $f$ and $X$.

\medskip
\noindent
{\bf Assumption RV$_\alpha(D,\nu)$.} {\em Let $D \subset \overline\rr^d$ be a closed cone in $\overline \rr^d$. 
We suppose that $X\in RV_\alpha(\{a_n\}, D,\nu)$, that is,  $X$ is regularly varying on $\overline\rr^d \setminus D$ with index 
$\alpha>0$ and some Radon measure $\nu$ that does not charge infinite lines, i.e.\ $\nu(\overline\rr^d\setminus(\rr^d\cup D)) = 0$.}

\medskip
\noindent{\bf Assumption H.}
{\em Let $f: \overline\rr^d \to [0,\infty]$ be Borel measurable, such that $f(x)<\infty,\ x\in\rr^d$, 
$f(0) = 0$, and 1-homogeneous, that is, 
\begin{equation}\label{eq4}
 f(\lambda x)=\lambda f(x)\quad\text{for all $\lambda>0$ and $x\neq 0$.}
\end{equation}}

\medskip
We shall use in the sequel the following two conditions relating $f$ and $\nu$.

\medskip
\noindent
{\bf Assumption F.} {\em For all $\epsilon>0$, the set $\{ f >\epsilon \}$ is bounded away from $D$. Furthermore, for all 
compact $K\subset \overline\rr^d_D$, we have
\begin{equation}\label{e:f-bdd-below}
\inf_{x\in K} f(x) > 0.
\end{equation}
}

\begin{remark} \label{rem:f-and-D} Assumption F implies that $\{f=0\} = D$. Indeed, $\{f>0\} = \cup_{\epsilon >0} \{f>\epsilon \} \subset  \overline\rr^d \setminus D$ and 
hence $D\subset \{f=0\}$. On the other hand for all $x\in \overline \rr^d\setminus D$, there exists a compact $K\subset \rr^d\setminus D$ such that $x\in K$
and thus $f(x)>0$, by \eqref{e:f-bdd-below}. This shows that $\{f=0\}\subset D$ and hence $\{f=0\} = D.$
\end{remark}

\medskip
\noindent
{\bf Assumption C.} We have $\nu( \overline{{\rm Disc}(f)} ) = 0$, where $\overline{{\rm Disc}(f)}$ denotes 
the closure in $\overline \rr^d_D$ of the set of all discontinuity points of $f$. 

\medskip
We fix some polar coordinates as in Definition \ref{d:polar-coors} so that the map
$x\mapsto (\tau,\theta)$ is a homeomorphism between the spaces $\rr^d_D$ and $(0,\infty)\times S$.
Recall the disintegration formula from Fact \ref{fact:disintegration}:
\begin{equation}\label{eq3a}
\nu(A)=\int_{S}\int_0^\infty 1_{A}(\tau\theta)\, \frac{\alpha d\tau}{\tau^{\alpha+1}}\,  \sigma(d\theta),
\end{equation}
where $\sigma= \sigma_S$ is the unique finite {\em spectral measure} of $\nu$, relative to these polar coordinates.

\medskip
\begin{remark} By the homogeneity of the function $f$, we have
\begin{equation}\label{e:f-polar}
 f(x) = \tau(x) f_0(\theta(x) ), 
\end{equation}
where $f_0: S \to (0,\infty)$ may be viewed as the {\em angular} part of $f$. By \eqref{eq3a}, 
Assumption C is equivalent to $\sigma(\overline{{\rm Disc}(f_0)}) =0$, where $S$ is equipped with the 
relative topology. This means that the atoms of the spectral measure $\sigma$ do not coincide with discontinuity 
points of the angular component $f_0$. 
\end{remark}

\begin{remark} Assumptions F and C are clearly fulfilled if $f:\overline \rr^d \to [0,\infty]$ is continuous and such that 
$D = \{f= 0\}$. In view of \eqref{e:f-polar}, however, interesting discontinuous homogeneous functions can be constructed that should 
be covered by a limiting theory.  This motivates the more general Assumption C.
\end{remark}
 
\begin{theorem}\label{thm1} Assumptions $RV_\alpha(D,\nu)$, H, F and C imply
\begin{equation}\label{eq5}
 a_n^{-1}X_{k(n)}  \Longrightarrow Y\quad\text{as $n\to\infty$}
\end{equation}
where $Y$ is a random vector taking values in $\overline\rr^d\setminus D$, with distribution
\begin{equation}\label{eq6}
P_Y(dx)=e^{-Cf(x)^{-\alpha}}\,\nu(dx)
\end{equation}
where
\begin{equation}\label{eq6-C}
 C:=\nu(\{ z\, :\, f(z)>1 \}) = \int_{S}f(\theta)^{\alpha}\,\sigma(d\theta) <\infty.
 \end{equation}
 The random vector $Y$ is proper, i.e.\ takes values in $\rr^d\setminus D$ since by assumption $\nu(\overline\rr^d\setminus(\rr^d\cup D)) = 0$, that is,
 $\nu$ does not charge points on the infinite lines.
 \end{theorem}

\begin{proof} 
 By the upper bound in \eqref{eq1-new} of Lemma \ref{lem1}, for any measurable set $A\subset\rr^d$, 
\begin{eqnarray} \label{e:thm1-1}
 P(a_n^{-1} X_{k(n)}\in A) &\le & n\int_{a_nA} P(f(X)\le f(x))^{n-1}P_X(dx)\nonumber\\ 
 &=& \int_{A} {\Big(} 1- \frac{nP(f(X)> f(a_nx))}{n}{\Big)}^{n-1} nP_{a_n^{-1}X} (dx) = : \int_A h_n^{+} d\nu_n,
\end{eqnarray}
where $\nu_n (dx) :=nP_{a_n^{-1}X} (dx)$.  

Similarly, using the lower bound in \eqref{eq1-new} along with the established \eqref{e:thm1-1}, we obtain
\begin{equation}\label{e:thm1-sandwich}
  \int_{A} h_n^{-} d\nu_n\le P(a_n^{-1} X_{k(n)}\in A) \le \int_{A} h_n^{+} d\nu_n.
\end{equation}
where
\begin{equation}\label{e:thm1-1.5}
 h_n^{-}(x):= {\Big(} 1- \frac{nP(f(X)\ge  f(a_nx))}{n}{\Big)}^{n-1}.
\end{equation}
We will show that the two measures $h_n^\pm d\nu_n$ sandwiching the law of 
$a_n^{-1} X_{k(n)}$ in \eqref{e:thm1-sandwich} converge to the same limit, which will 
ultimately yield \eqref{eq5}.  We will first present the intuition and then make the argument precise.

\medskip
By the homogeneity of $f$ and \eqref{e:d:RV-D}, as $n\to\infty$,
\begin{eqnarray}\label{e:thm1-alt}
 nP(f(X)> f(a_nx))  &\equiv& nP(f(X) > a_n f(x) ) \nonumber\\
  &=& n P(a_n^{-1} X \in \{ y \, :\, f(y)>f(x) \}) \nonumber\\
  &\longrightarrow&  \nu(\{ y\, :\, f(y)>f(x)\}).
\end{eqnarray}
This is true, provided $\{ y\, :\, f(y)>f(x)\}$ is a $\nu$-continuity set, which is bounded away from $D$. 
If this is the case, for $h_n^{+}$ in \eqref{e:thm1-1}, we have
\begin{equation}\label{e:thm1-2}
 h_n^{+}(x)\longrightarrow h^{+}(x):= e^{- \nu(\{ y\, :\, f(y)>f(x)\})} \equiv e^{ -  f(x)^{-\alpha} \nu(\{f>1\}) },\ \ \mbox{ as }n\to\infty,
\end{equation}
where in the last relation we used the homogeneity of $f$ and the scaling property of $\nu$.
Under similar conditions, for $h_n^{-}$ in \eqref{e:thm1-1.5}, we obtain
\begin{equation}\label{e:thm1-2.5}
  h_n^{-}(x)\longrightarrow h^{-}(x):= e^{- \nu(\{ y\, :\, f(y)\ge f(x)\})}\equiv e^{ -  f(x)^{-\alpha} \nu(\{f\ge 1\}) },\ \ \mbox{ as }n\to\infty.
\end{equation}

The limit functions $h^{+}$ and $h^{-}$ coincide. Indeed, by the homogeneity of $f$ and the scaling property 
of $\nu$, for all $c>0$
$$
 \nu(\{f\ge c\}) - \nu(\{f>c\}) = \nu(\{f = c\}) = c^{-\alpha}\nu(\{f=1\}).
$$
The sets $\{f=c\},\ c>0$ are, however, disjoint. This, since $\{f>\epsilon\} = \cup_{c>\epsilon} \{f=c\}$ has finite
$\nu$-measure for all $\epsilon>0$ (Assumption F), implies that $\nu(\{f=c\}) = c^{-\alpha} \nu(\{f=1\}) = 0$, for all $c>0$.  
Consequently, $\nu(\{f>1\}) = \nu(\{f\ge 1\})$ and 
\begin{equation}\label{e:thm1-h}
h^{+}(x) = h^{-}(x) =: h(x) = e^{- C f(x)^{-\alpha}},\ \mbox{ for all }x\in \rr^d\setminus D.
\end{equation} 

Recall that by Assumption RV$_\alpha(D,\nu)$, we have $\nu_n\stackrel{v}{\to} \nu,\ n\to\infty$. Hence, Relations \eqref{e:thm1-2}, \eqref{e:thm1-2.5} and \eqref{e:thm1-h}
suggest that the probability measures in \eqref{e:thm1-sandwich} converge to the same measure $P_Y(dx) = h(x) \nu(dx)$. We will show this is indeed the case 
by using Lemmas \ref{l:weak-convergence} and \ref{l:f(X)-rv}, given in the Appendix.

\medskip
Since $\nu_n\stackrel{v}{\to} \nu$, as $n\to\infty$, by Lemma \ref{l:weak-convergence}, the measures in the right-hand side of \eqref{e:thm1-sandwich} 
converge to $h(x) \nu(dx)$, provided  \eqref{e:thm1-2}  holds uniformly in $x$ over all compact subsets of
$\overline\rr^d_D$. This is true, if \eqref{e:thm1-alt} is valid uniformly in $x$ over $K$, for all compacts $K\subset \overline\rr^d_D$.
Note that by \eqref{e:f-bdd-below}, the function $f(x)$ is uniformly bounded away from $0$ over the compact $K$.
Therefore, Lemma \ref{l:f(X)-rv} (iii) applied with $y:=f(x)$ to \eqref{e:thm1-alt} yields the desired uniform convergence. The argument showing 
that the left-hand side in \eqref{e:thm1-sandwich} converges to $h(x) \nu(dx)$ is similar.

\medskip
{\it To complete the proof}, it remains to show that, in the limit, no mass is lost at infinity, and the measure $P_Y(dx) = h(x)\nu(dx)$ given by 
\eqref{eq6} and \eqref{eq6-C} is a valid probability distribution on $\rr^d\setminus D$.  Note first that by assumption $\nu$ does not put any mass 
on the infinite hyperplanes, i.e.\  $\nu(\overline\rr_D^d\setminus \rr_D^d) = 0$. Thus, the support of the measure $P_Y$ is confined to 
$\rr^d\setminus D$.

Using the 1-homogenity of $f$ and the scaling property of $\nu$, we obtain for all $x\in\rr^d\setminus D$
\begin{equation*}
\nu\{ y : f(y)>f(x)\}  
= \nu\{z : f(z)>1\} f(x)^{-\alpha} = C f(x)^{-\alpha}.
\end{equation*}
This shows that $h(x) = e^{-\nu \{ y : f(y)>f(x)\}} = e^{-C f(x)^{-\alpha}}$. Next, we establish the second 
expression for $C$ in \eqref{eq6-C}. Using the disintegration formula \eqref{eq3a}, we get
\begin{equation*}
\begin{split}
C = \nu\{z : f(z)>1\} &= \int_{S}\int_0^\infty 1_{\{z : f(z)>1\}}(\tau\theta)\,\frac{\alpha d\tau}{\tau^{\alpha+1}}\,\sigma(d\theta) \\
& = \int_{S}\int_{f(\theta)^{-1}}^\infty \frac{\alpha d\tau}{\tau^{\alpha+1}}\,\sigma(d\theta) \\
& = \int_{S}f(\theta)^{\alpha}\,\sigma(d\theta)<\infty.
\end{split}
\end{equation*}
As a by-product of the above computation, we obtain that the last integral is finite since it 
equals $C\equiv \nu\{f>1\}<\infty$, by Assumption F.

Finally, using \eqref{eq3a} again, we verify that \eqref{eq6} integrates to one
\begin{equation*}
\begin{split}
\int_{\rr^d\setminus D}e^{-Cf(x)^{-\alpha}}\,\nu(dx) &= \int_{S}\int_{(0,\infty)} e^{-Cf(\tau \theta)^{-\alpha}}\,\frac{\alpha d\tau}{\tau^{\alpha+1}}\, \sigma(d\theta) \\
&= \int_{S}\int_{(0,\infty)} e^{-C\tau^{-\alpha} f(\theta)^{-\alpha}}\,\frac{\alpha d\tau }{\tau^{\alpha+1}}\, \sigma(d\theta) \\
&=C^{-1}\int_{S}f(\theta)^{\alpha }\, \sigma(d\theta)\int_0^\infty e^{-u}\,du = 1,
\end{split}
\end{equation*}
where in the last line, we used the change of variables $u:= C\tau^{-\alpha} f(\theta)^{-\alpha}$ and the already established Relation \eqref{eq6-C}.
\end{proof}

\begin{remark} Suppose that $X \in RV_\alpha(D,\nu)$ and let $f$ be a {\em continuous} non-negative homogeneous function.
The requirement that $D = \{f=0\}$ following from Assumption F can be circumvented. 
Indeed, if $ D \subset \widetilde D := \{f=0\},$
then $X\in RV_\alpha(\widetilde D,\widetilde \nu)$, where $\widetilde \nu := \nu\vert_{\overline\rr^d\setminus \widetilde D}$ is the
restriction of $\nu$. Then, by the continuity of $f$, Assumptions F, H and C hold and hence  \eqref{eq5} is valid over the restricted space $\overline\rr^d_{\widetilde D}$.

It may happen, however, that ${\rm supp}(\nu) \subset \widetilde D$ and so trivially $\widetilde \nu = 0$. As seen in Example \ref{ex:Pareto-Dirichlet}, above,
essentially different measure of regular variation may arise on the restricted cone $\overline\rr^d\setminus\widetilde D$.
This shows that, in general, when focusing on $f$-implicit extremes, the natural domain of regular variation 
is $\overline\rr^d\setminus \{f=0\}.$ Finally, if $\{f=0\} \subset D$, the above argument may fail since $X$ may not be regularly 
varying in the larger cone $\overline \rr^d\setminus\{f=0\}$. 
\end{remark}

\begin{remark}\label{rem:continuity-at-infinity} It is important to note that Theorem \ref{thm1} may fail for {\em continuous} 
$f : \rr^d\to [0,\infty)$ that are, however, not continuous on the extended space $\overline \rr^d$.
Indeed, consider for example the $1-$homogeneous function $f(x_1,x_2):= \sqrt{x_1x_2},\ x_1,x_2\ge 0$,
defined as $0$ elsewhere.  Let also $X = (1/U_1\ 1/U_2)$ be as in Example \ref{ex:Pareto-Dirichlet} above,
where $U_1$ and $U_2$ are independent Uniform$(0,1)$. We have that $X \in RV_{\alpha}(\{n^2\},D,\nu)$,
where $\alpha = 2$, $D := \overline \rr^2 \setminus (0,\infty]^2$, and
$$
\nu(dx_1 dx_2) = x_1^{-2} x_2^{-2} dx_1 dx_2\ \ \mbox{ on }(0,\infty)^2.
$$
One may be tempted to conclude that \eqref{eq5} holds. Notice that for all $C>0$, we have
$$
 \int_{(0,\infty)^2} e^{-C f^{-\alpha}(x) } \nu(dx) = \int_0^\infty \int_0^\infty e^{-C x_1^{-1} x_2^{-1}} x_1^{-2} x_2^{-2} dx_1dx_2 = 
 \int_0^\infty \int_0^\infty e^{-C u_1 u_2} du_1 du_2 = \infty
$$
and therefore \eqref{eq6} does not define a valid probability distribution.
\end{remark}

\begin{defi} The limits arising in \eqref{eq5} will be referred to as {\em $(f,\nu)$-implicit extreme value} laws.
\end{defi}

We have the following probabilistic representation. Recall that a random variable 
$Z$ is said to be standard $\alpha$-Fr\'echet ($\alpha>0$), if $P(Z\le x) = e^{-x^{-\alpha}},\ x>0$.

\begin{prop}\label{p:thm1-structure}
The random vector $Y$ in $\rr^d\setminus D$ has an $(f,\nu)$-implicit extreme value law if and only if for some
measurable $g:S \to [0,\infty)$ with $\int_{S} g^\alpha(\theta) \sigma(d\theta) =1$, 
\begin{equation}\label{e:p:thm1-structure}
 Y \eqd Z \frac{\Theta}{g(\Theta)},
\end{equation}
where $Z$ standard $\alpha$-Fr\'echet and $\Theta$ is an independent of $Z$ random vector taking values in $S$
and having distribution $\sigma_g(d\theta):= g^{\alpha}(\theta)\sigma(d\theta)$. 

Moreover, the function $g$ in \eqref{e:p:thm1-structure} is unique, modulo $\sigma$-null sets
and, in the context of Theorem \ref{thm1}, it is given by $g(\theta) = C^{-1/\alpha}f(\theta),\ \theta\in S$. 
(Note that $P(g(\Theta) =0) =0$ and so \eqref{e:p:thm1-structure} is well--defined.)
\end{prop}
\begin{proof} ($\Rightarrow$) Suppose first that $Y$ is a $(f,\nu)$-implicit extreme value, that is, \eqref{eq5} holds for
some $\alpha>0$ and $1$-homogeneous function $f$. In view of the disintegration formula \eqref{eq3a} of the
measure $\nu$, the law of $Y$ in \eqref{eq6} has the following representation in polar coordinates
$$
P_Y(d \tau \sigma(d\theta)) = e^{-C f(\tau \theta)^{-\alpha}} \frac{\alpha d\tau }{\tau ^{\alpha+1} } \sigma(d\theta).
$$
Consider an arbitrary `rectangle' in polar coordinates, i.e.\ $A_{r,B} := 
\{ x\in \rr^d\setminus D \, :\, \tau \le r,\ \theta(x) \in B\}$, for $r>0$ and
a Borel set $B\subset S$. We have that
\begin{eqnarray} \label{e:p:thm1-structure-1}
P(Y\in A_{r,B}) &=& \int_{S}\int_{0}^\infty 1_{A_{r,B}}(\tau\theta) 
 e^{-C f( \tau\theta)^{-\alpha}} \frac{\alpha d\tau }{\tau^{\alpha+1}} \sigma(d\theta)\nonumber\\
 &=& \int_{B}\int_0^{r} e^{-Cf(\theta)^{-\alpha} \tau^{-\alpha}} \frac{\alpha d\tau}{\tau^{\alpha+1}} \sigma(d\theta)\nonumber\\
 & = & \int_{B} \int_{{C f(\theta)^{-\alpha} r^{-\alpha}}}^\infty e^{-u} du\ C^{-1} f(\theta)^{\alpha} \sigma(d\theta), 
\end{eqnarray}
where in the second relation we used the homogeneity of $f$ and in the last relation, we made the change of variables 
$u:= C f(\theta)^{-\alpha} \tau^{-\alpha}$. Note that the inner integral in \eqref{e:p:thm1-structure-1} equals 
$$
  e^{-C f(\theta)^{-\alpha} r^{-\alpha}} = P(Z C^{1/\alpha} f(\theta)^{-1}\le r),
$$
for a standard $\alpha$-Fr\'echet variable $Z$. We therefore obtain
\begin{equation}\label{e:p:thm1-structure-2}
 P(Y\in A_{r,B}) \equiv P( \tau(Y) \le r,\ \theta(Y) \in B) = \int_{B} P(  Z C^{1/\alpha} f(\theta)^{-1} \le r )  \widetilde \sigma(d\theta),
\end{equation}
where $\widetilde \sigma(d\theta) := C^{-1} f(\theta)^{\alpha} \sigma(d\theta).$ Observe that the choice of the 
constant $C$ ensures that $\widetilde \sigma (d\theta)$ is a probability distribution on $S$.  

Suppose now that $\Theta$ is an independent of $Z$,  $S$-valued random vector with probability distribution 
$\widetilde \sigma$. Using the independence of $Z$ and $\Theta$, we see that the right-hand side of \eqref{e:p:thm1-structure-2} equals
$P( Z C^{1/\alpha} f(\Theta)^{-1} \in A_{r,B})$. This shows that the distributions of $Y$ and 
$Z \Theta/ g(\Theta)$ coincide on the class of sets $A_{r,B},\ r>0,\ B\in {\mathcal B}(S)$,
where
 $$
  g(\theta) := C^{-1/\alpha} f(\theta),\ \theta\in S. 
 $$
Since the latter class is a $\pi$-system, generating the Borel $\sigma$-algebra on $\rr^d\setminus D$, the $\pi$-$\lambda$ theorem 
shows that \eqref{e:p:thm1-structure} holds.

\medskip
($\Leftarrow$)  Conversely, for an arbitrary non-negative measurable function $g:S\to [0,\infty)$ with 
$\int_{S} g(\theta)^\alpha \sigma(\theta) = 1$, let $f(x) := \tau(x) g(x/\tau(x)) \equiv \tau g(\theta)$ 
be a $1$-homogeneous function. Consider the random vector
$$
X:= Z \frac{\Theta}{g(\Theta)},
$$
where $Z$ and $\Theta$ are independent with standard $\alpha$-Fr\'echet and $\sigma_g$ laws, respectively. Let
$(Z_i,\Theta_i)$, $1\le i\le n$ be independent copies of $(Z,\Theta)$. By homogeneity
$$
f(X_i) = f{\Big(} Z_i{\Theta_i\over g(\Theta_i)} {\Big)} = Z_i \frac{g(\Theta_i)}{g(\Theta_i)} = Z_i,\ 1\le i\le n. 
$$
That is, $f(X_i),\ 1\le i\le n$ are iid $\alpha$-Fr\'echet, that do not depend on the directions $\Theta_i = \theta(X_i)$ 
of the vectors $X_i$. Hence the random variable $k(n)$ in \eqref{e:k(n)-def} is independent of $\Theta_i,\ 1\le i\le n$ and
\begin{equation}\label{e:imp-max-stab-motivation}
 X_{k(n)} = Z_{k(n)}  \frac{\Theta_{k(n)}}{g(\Theta_{k(n)})} \stackrel{d}{=} {\Big(} \bigvee_{i=1}^n Z_i {\Big)} 
 \frac{\Theta_{1}}{g(\Theta_{1})} \eqd n^{1/\alpha} X,
\end{equation}
where in the last relation we used the fact that $\vee_{i=1}^n Z_i \eqd n^{1/\alpha} Z$.
 Relation \eqref{e:imp-max-stab-motivation} shows that \eqref{eq5} holds trivially in this case, where
 $a_n:=n^{1/\alpha}$ and $Y \eqd X$. That is, any $Y$ as in \eqref{e:p:thm1-structure} can be 
 a limit in \eqref{eq5}.

To complete the proof, it remains to show that the function $g$ in \eqref{e:p:thm1-structure} is unique.
By letting $r\to\infty$ in \eqref{e:p:thm1-structure-1}, we see that
$$
P(\theta(Y) \in B) = \int_B C^{-1} f(\theta)^{\alpha} \sigma(t\theta),
$$
for all Borel $B\subset S$. This uniquely identifies $g$ as $g(\theta) = C^{-1/\alpha} f(\theta),\ \theta\in S$, 
modulo $\sigma$-null sets.
 \end{proof}

\begin{remark} \label{rem:composition}
 Observe that \eqref{e:k(n)-def} remains unchanged if $f$ is replaced by $\psi\circ f$, for any monotone
 strictly increasing function $\psi$. This shows that the result of Theorem \ref{thm1} automatically
 extends to functions $f$ such that $\psi^{-1}\circ f$ is $1$-homogeneous and satisfies the assumptions of the 
 theorem.
\end{remark}

The following result shows that $(f,\nu)$-implicit max-stable laws appearing in Theorem \ref{thm1}
are also in the class $RV_\alpha(D,\nu)$, as expected.

\begin{cor}\label{c:thm1} If  $Y$ is an $(f,\nu)$-implicit extreme value random vector as in Theorem \ref{thm1}, then
$Y\in RV_\alpha(\{n^{1/\alpha}\}, D, \nu)$. In fact, for all $x>0$, we have 
\begin{equation}\label{e:c:thm1}
 nP(n^{-1/\alpha} \tau(Y) > x) \mathop{\longrightarrow}_{n\to\infty} \sigma_S(S) x^{-\alpha}\ \ \mbox{ and }\  \ P(\theta(Y) \in \cdot \vert \tau (Y) > u)
 \mathop{\stackrel{TV}{\longrightarrow}}_{u\to\infty} \sigma_0(\cdot),
\end{equation}
where $\sigma_S$ is as in \eqref{e:fact:disintegration} and $\sigma_0(\cdot) = \sigma_S(\cdot)/\sigma_S(S)$.
\end{cor}
\begin{proof} The result readily follows from the Breiman-type Lemma \ref{l:Breiman}, above, applied to $X:= Z V$, where
$V := \Theta/g(\Theta)$. Note that now the law $P_V$ of $V$ is concentrated on the deformed 
unit sphere $\{\theta/g(\theta)\, :\, \theta \in S\}$. 
\end{proof}

\begin{remark} If $X\in  RV_\alpha(\{a_n\}, D, \nu)$, then $X\in RV_\alpha(c\{a_n\},D, c^{-\alpha} \nu)$, for all 
$c>0$. Thus, upon rescaling, we can always ensure that the spectral measure is a probability measure.
\end{remark}

The next result shows the uniqueness of the stochastic representation of the implicit extreme value laws.

\begin{cor}\label{c:uniqueness} The representation in \eqref{e:p:thm1-structure} is unique. More precisely,
if $(\alpha,g,\sigma_g)$ and $(\widetilde \alpha, \widetilde g, \widetilde \sigma_{\widetilde g})$ are two triplets parameterizing 
the right-hand side therein, then $\alpha = \widetilde \alpha$, $\sigma_S = \widetilde \sigma_S$, and $g = \widetilde g$ (mod $\sigma_S$). 
\end{cor} 
\begin{proof}  Suppose that
\begin{equation}\label{e:c:uniqueness}
  Y\stackrel{d}{=} Z \frac{\Theta}{g(\Theta)} \stackrel{d}{=} \widetilde Z \frac{\widetilde \Theta}{g(\widetilde \Theta)},
\end{equation}
where the tilded quantities correspond to the stochastic representation as in 
\eqref{e:p:thm1-structure} with parameters $(\widetilde \alpha, \widetilde g, \widetilde \sigma_{\widetilde g})$. 
By Corollary \ref{c:thm1}, we have $\alpha = \widetilde \alpha$ and $\sigma_S = \widetilde \sigma_S$. 
On the other hand, by
\eqref{e:c:uniqueness},  
$$
 \theta(Y) \stackrel{d}{=} \theta{\Big(} Z \frac{\Theta}{g(\Theta)} {\Big)} = \Theta \stackrel{d}{=} \widetilde \Theta,
$$
and hence $\sigma_g = \widetilde\sigma_{\widetilde g}$, which yields $g = \widetilde g$ (mod $\sigma_S\equiv \widetilde \sigma_S$).
\end{proof}

\begin{remark} What happens with the stochastic representation in \eqref{e:p:thm1-structure} under another set of
polar coordinates $(\tau^*,\theta^*)$? Let $S^* = \{ \tau^* =1\}$ and define 
the natural bijection $\lambda: S\to S^*$, where $\lambda(\theta) = \theta/\tau^* (\theta)$ is simply a rescaled version of 
the vector $\theta$. Suppose that \eqref{e:p:thm1-structure} holds and observe that by homogeneity,
\begin{equation}\label{e:new-coors}
Z \frac{\Theta}{g(\Theta)} = Z \frac{\Theta/\tau^* (\Theta)}{g(\Theta/\tau^*(\Theta))} 
=: Z\frac{\Theta^*}{g(\Theta^*)},\ \ \mbox{ surely (not just almost surely)}.
\end{equation}
Observe that $Z$ and $\Theta^*:=\Theta/\tau^*(\Theta)$ are independent and $\Theta^*$ takes values in the new unit sphere $S^*$. 
The uniqueness of the stochastic representation (Corollary \ref{c:uniqueness}) then implies that the right-hand side 
\eqref{e:new-coors} provides the stochastic representation of $Y$ with respect to the new polar coordinates.

It is remarkable that the relationship between the two stochastic representations is deterministic. That is, the two involve
the same $\alpha$-Fr\'echet random variable and {\em the same} directional component vector 
$\Theta/g(\Theta) \equiv \Theta^*/g(\Theta^*)$. This shows that the representation in \eqref{e:p:thm1-structure}
does not depend on the choice of polar coordinates.
\end{remark}

\section{Implicit max-stable laws and their domains of attraction}\label{sec:DA}

Relation \eqref{e:imp-max-stab-motivation} in the proof of Proposition \ref{p:thm1-structure} suggests the following notion of
{\em $f$-implicit max-stable} distributions. 

\begin{defi} \label{f-imp-def}An $\rr^d$-valued random vector $X$ is said to be implicit max-stable with respect to a homogeneous
function $f$, or simply {\it $f$-implicit max-stable}, if for all $n$, there exist $a_n>0$  such that
\begin{equation}\label{eqp1}
 a_n^{-1} X_{k(n)} \stackrel{d}{=} X,
\end{equation}
where $k(n)$ is as in \eqref{e:k(n)-def}, and  $X_i,\ 1\le i\le n$ are independent copies of $X$.
\end{defi}


By \eqref{e:imp-max-stab-motivation}, all $(f,\nu)$-implicit extreme value laws are also $f$-implicit max-stable. 
Under the mild additional assumption that $f$ is continuous, the converse is also true, as shown next.

\begin{theorem}\label{thm_f} Let $f : \overline\rr^d \to [0,\infty]$ be non-negative, continuous and $1$-homogeneous 
function such that $f(x)<\infty,\ x\in \rr^d$. Then, a distribution is strictly $f$-implicit max-stable if and only if it 
is a $(f,\mu)$-implicit extreme value distribution, where $\mu$ is supported on $\rr^d\setminus \{f=0\}$ and 
satisfies the scaling property  $\mu(\lambda \cdot)= \lambda^{-\alpha}\mu(\cdot)$ for all $\lambda>0$ and some 
$\alpha>0$.
\end{theorem}

\begin{proof} ($\Leftarrow$): By the continuity of $f$, the assumptions of Theorem \ref{thm1} hold, and the claim follows from
Relation  \eqref{e:imp-max-stab-motivation} in the proof of Proposition \ref{p:thm1-structure}.

\medskip
($\Rightarrow$): Assume that \eqref{eqp1} holds. By the homogeneity of $f$ and the definition of $k(n)$, Relation \eqref{eqp1} implies 
\[ 
 a_n^{-1}f(X_{k(n)})=a_n^{-1}\max\bigl\{f(X_1),\dots,f(X_n)\bigr\}\eqd f(X) 
\]
for all $n\geq 1$. Then, $f(X)$ is a max-stable random variable supported on $[0,\infty).$ Hence, by classical extreme value theory, we know that 
$f(X)$ has an $\alpha$-Fr\'echet distribution and $a_n=n^{1/\alpha}$ for some $\alpha>0$. Thus, there exists a constant $C>0$ such that
\[ 
 P\{f(X)\leq x\}=e^{-Cx^{-\alpha}}\quad\text{for all $x>0$.} 
 \]
This implies in particular that $f(X)$ has continuous distribution and by Lemma \ref{lem1}, 
\[ 
 P\bigl\{X_{k(n)}\in A\bigr\}=n\int_A e^{-(n-1)Cf(x)^{-\alpha}}\,P_X(dx). 
\]
Note also that $P_X\{f=0\} = P\{f(X) = 0\} = 0$. Thus, the mass of 
$P_X$ is concentrated on $\rr^d\setminus D$, where $D:= \{f=0\}$. Without loss of generality, in the rest of the proof, we shall consider all 
measures over $\rr^d\setminus D$.

By \eqref{eqp1} with $a_n=n^{1/\alpha}$ and using the homogeneity of $f$, we obtain that for all $n\geq 1$
\begin{equation}\label{e:thm_f-.5} 
 P\{X\in A\}=P\bigl\{n^{-1/\alpha}X_{k(n)}\in A\bigr\}=\int_A e^{-(1-n^{-1})Cf(x)^{-\alpha}}\,n\cdot P_{n^{-1/\alpha}X}(dx),
\end{equation}
for all measurable $A \subset \rr^d\setminus D.$ We will show that this implies
\begin{equation}\label{e:thm_f-1} 
 \mu_n(dx):= n\cdot P_{n^{-1/\alpha}X}(dx)\overset{v}{\longrightarrow}\mu(dx) \quad\text{as $n\to\infty$} 
\end{equation}
for some Radon measure $\mu$ on $\rd\setminus D$.

Indeed, \eqref{e:thm_f-.5} means that $g_n(x) := e^{-(1-n^{-1})Cf(x)^{-\alpha}}$ is the Radon-Nikodym derivative of $P_X$ with 
respect to $\mu_n$. Since $f(x)>0$, we have $g_n(x)>0$, for all $x\in\rr^d\setminus D$ and hence 
$\mu_n \ll P_X$. Thus, letting $h_n:= g_n^{-1} \equiv d \mu_n/d P_X$, we obtain
\[
  n\cdot P_{n^{-1/\alpha}X}(A) \equiv \mu_n(A) =\int_{A} h_n(x) P_X(dx).
\]
Observe that $h_n(x)=e^{(1-n^{-1}) C f(x)^{-\alpha}}$ converges to $h(x):= e^{Cf(x)^{-\alpha}},$ as $n\to\infty$, uniformly over all compacts in $\rr^d\setminus D$. Therefore,
by applying Lemma \ref{l:weak-convergence} with $\mu_n$, $h_n$ and $h$ as above to the trivial case $\nu_n \equiv \nu := P_X$, we obtain
\eqref{e:thm_f-1}, where in fact
\begin{equation}\label{e:thm_f-2} 
  \mu(A) = \int_{A} h(x) P_X(dx)  \equiv \int_A e^{C f(x)^{-\alpha} } P_X(dx).
\end{equation}

Relation \eqref{e:thm_f-1} means that $X \in RV_\alpha(\{n^{1/\alpha}\},D,\mu)$.  Furthermore, since 
$f(x)>0,$ for all $x\in\rr^d\setminus D$,  Relation \eqref{e:thm_f-2} is equivalent to
\begin{equation}\label{e:thm_f-3} 
 P\{ X \in A\} =\int_A e^{-Cf(x)^{-\alpha}}\,\mu(dx),
 \end{equation}
for all Borel sets $A\subset \rr^d\setminus D$, showing that $X$ has a $(f,\mu)$-extreme value law. Notice that
as in the proof of Theorem \ref{thm1}, the constant $C$ satisfies \eqref{eq6-C}. 
\end{proof}

\begin{defi}
Fix $f : \overline\rr^d \to [0,\infty]$ as in Theorem \ref{thm_f}. We say that a random vector belongs to the $f$-implicit domain of attraction 
of a (necessarily) $f$-implicit max-stable random vector $Y$, if there exist $a_n>0$ such that
\begin{equation}\label{eqp2}
a_n^{-1}X_{k(n)}\Longrightarrow Y\quad\text{as $n\to\infty$}
\end{equation}
where $k(n)$ is as in \eqref{e:k(n)-def} and $X_1,\dots,X_n$ are i.i.d. as $X$. We write $X\in\DOA_f(Y)$ in this case.
\end{defi}

\begin{theorem}\label{thm:DOA_char} Let $f : \overline\rr^d \to [0,\infty]$ be non-negative, continuous and $1$-homogeneous function, such that $f(x)<\infty,\ x\in \rr^d$.
Then, $X\in\DOA_f(Y)$ if and only if $X\in \RV_\alpha(\{f=0\},\mu)$ for some $\alpha>0$.
\end{theorem}

\begin{proof} ($\Leftarrow$): Theorem \ref{thm1} shows that if $X\in \RV_\alpha(\{f=0\},\nu)$, then $X\in\DOA_f(Y)$ and $Y$ has a $f$-implicit 
max-stable law by Theorem \ref{thm_f}. 

($\Rightarrow$): Assume now that \eqref{eqp2} holds. Then, by the continuous mapping theorem, we have 
$$
  a_n^{-1} f(X_{k(n)})=a_n^{-1} \max\bigl\{f(X_1),\dots,f(X_n)\bigr\}\Longrightarrow f(Y)\quad\text{as $n\to\infty$.} 
$$
 This shows that $f(X)$ belongs to the domain of attraction of the (necessarily) $\alpha$-Fr\'echet random variable $f(Y)$. 
 Thus, $a_n$ is regularly varying with index $1/\alpha$ and there exists a constant $C>0$ such that
\begin{equation}\label{eqp3}
 P\bigl\{a_n^{-1} f(X)\leq y\bigr\}^{n-1}\to e^{-Cy^{-\alpha}}\quad\text{as $n\to\infty$}
\end{equation}
uniformly in $y>0$. In view of Lemma \ref{lem1}, we then get
\begin{equation}\label{eqp4}
 \int_A g_n^{-} (x) \mu_n(dx) \le P\bigl\{a_n^{-1}X_{k(n)}\in A\bigr\}\le   \int_A g_n^{+} (x) \mu_n(dx),
\end{equation}
where $\mu_n(\cdot) := nP(a_n^{-1} X \in \cdot)$, and where
$$
 g_n^-(x)  =  P\bigl\{a_n^{-1} f(X)< f(x)\bigr\}^{n-1}\quad\mbox{ and }\quad  g_n^+(x)  =  P\bigl\{a_n^{-1} f(X)\leq f(x)\bigr\}^{n-1}.
$$
Notice that by \eqref{eqp3},
\begin{equation}\label{e:thm:DOA_char-2}
 g_n^\pm (x) \mathop{\longrightarrow}_{n\to\infty} g(x) := e^{-C f(x)^{-\alpha}},\ \ \mbox{ for all }x\in \rr^d\setminus D \equiv \rr^d \setminus\{f=0\}. 
\end{equation}
We will show that \eqref{eqp4} and \eqref{e:thm:DOA_char-2} imply
\begin{equation}\label{eqp5}
 \mu_n(dx) \equiv n\cdot P_{a_n^{-1} X}(dx)\overset{v}{\longrightarrow}\mu(dx),\ \ \mbox{ as }n\to\infty,
\end{equation}
for some Radon measure $\mu$ on $\rd\setminus\{f=0\}$. To this end, observe that it is enough to show that for all fixed compacts $K\subset 
\overline\rr^d\setminus\{f=0\}$, we have
\begin{equation}\label{eqp5.5}
  \mu_n( \cdot\cap K) \equiv n\cdot P_{a_n^{-1}X}(\cdot\cap K ) \overset{w}{\longrightarrow}\mu(\cdot\cap K),\ \ \mbox{ as }n\to\infty.
\end{equation}
Proceeding as in the proof of Theorem \ref{thm_f}, let $\nu_n(dx):= P\bigl\{a_n^{-1}X_{k(n)}\in dx \bigr\}$ and
$\nu:= P_Y$ be the laws of the left- and right-hand side in \eqref{eqp2}, respectively. Then, by
\eqref{eqp4}, we have
\begin{equation}\label{e:thm:DOA_char-2.1}
  g_n^{-} (x) \le \frac{d\nu_n}{d\mu_n} (x)  \le  g_n^+(x),\ \quad x\in \rr^d\setminus\{ f= 0\}.
\end{equation}
The continuity of $f$ over the compact $K$ implies $\inf_{x\in K } f(x)>0$. Thus, by Relation \eqref{e:thm:DOA_char-2} 
for all sufficiently large $n$, we have $\inf_{x\in K} g_n^{\pm}(x) >0$. 
This, in view of \eqref{e:thm:DOA_char-2.1}, shows that 
$\mu_n\vert_{K }\ll \nu_n\vert_{K}$, for all sufficiently large $n$ and hence
$$
 \mu_n(A\cap K ) = \int_{A\cap K} h_n(x) \nu_n(dx),\ \quad\ A\in {\mathcal B}(\rr^d\setminus\{f=0\}),
$$ 
where
$$
 \frac{1}{g_n^+(x) }\le    h_n(x) := \frac{d\mu_n}{d\nu_n}(x) \le \frac{1}{g_n^{-}(x)},\ \ x\in K.
 $$
By the uniformity of the convergence in \eqref{eqp3} and the continuity of $f$, we also have that the convergences in \eqref{e:thm:DOA_char-2} are
uniform over the compact $K$. This shows that $h_n$ converges to $h(x) := g^{-1}(x)= e^{C f(x)^{-\alpha}}$, uniformly in $x \in K$, as $n\to\infty$. 
Thus, Lemma \ref{l:weak-convergence} applied to the measures $\mu_n$, $\nu_n$ and $\nu = P_Y$, restricted to $K$, yields \eqref{eqp5.5}.
Since the choice of the compact $K$ was arbitrary, we obtain \eqref{eqp5}, where
$$
 \mu(A):= \int_A h(x) \nu(dx) \equiv \int_A e^{C f(x)^{-\alpha}} P_Y(dx).
$$
Relation \eqref{eqp5} and the fact that $a_n$ is regularly varying with index $1/\alpha$ imply 
that $\mu(\lambda \cdot)=\lambda^{-\alpha}\mu(\cdot)$ for all $\lambda>0$ and that $X\in RV_\alpha(\{a_n\},\{f=0\},\mu)$. 
\end{proof}

\begin{remark} In view of Corollary \ref{c:thm1}, $f$-implicit max-stable laws (for continuous $f$) are regularly 
varying and belong to their own domain of implicit attraction, as expected.
\end{remark}

\begin{remark} The continuity assumption in Theorem \ref{thm_f} can be relaxed.  Note that the continuity of $f$
is not used in the proof of the `only if' part and it is only used in the `if' part to justify the application of Theorem \ref{thm1}.
Therefore, one can merely suppose that $f$ satisfies the assumptions of the last theorem.
The continuity assumption in Theorem \ref{thm:DOA_char} can be similarly relaxed.
\end{remark}

\section{Implicit Order Statistics}\label{sec:o-stat}

 In this brief section we study the natural counterpart of order statistics relative to a given loss function $f$.
 Namely, suppose that $X_i,\ i=1,\ldots,n$ are independent copies of a vector $X$. Consider the order 
 statistics of the scalar sample of losses $\xi_i:= f(X_i),\ i =1,\ldots,n$. That is, let 
 $\{k(1;n),\cdots,k(n;n)\}$ be a permutation of $\{1,\ldots,n\}$ such that
 $$
 \xi_{k(1;n)} \equiv f(X_{k(1;n)}) \ge \xi_{k(2;n)}\equiv f(X_{k(2;n)}) \ge \cdots \ge \xi_{k(n;n)} \equiv f(X_{k(n;n)}).
 $$
 where, by convention, possible ties among the $\xi_i$'s are resolved by taking the 
 indices $k(\cdot;n)$ in an increasing order.
 We shall refer to $X_{k(i;n)},\ i=1,\ldots,n$ as to the {\em implicit order statistics} relative to the loss $f$.
 Observe that $X_{k(1;n)} \equiv X_{k(n)}$ is the {\em implicit maximum} defined in \eqref{e:X-k(n)-def} above.
 
 We will establish the asymptotic behavior of the implicit order statistics for homogeneous losses and regularly varying 
 $X$.  To this end, it is convenient to consider polar coordinates generated by the loss function.  Specifically, let
 $f:\overline \rr^d\to[0,\infty]$ be a continuous homogeneous loss function, such that
 $f(x) <\infty$ for all $x\in \rr^d$.  Let also $\nu$ be a Radon measure on $\overline \rr^d\setminus\{f=0\}$, such 
 that $\nu(\overline \rr^d\setminus \rr^d) = 0$ and
 $$
 \nu(\lambda \cdot) = \lambda ^{-\alpha} \nu(\cdot),\ \ \mbox{ for all }\lambda>0,
 $$
 with some exponent $\alpha>0$.
 
 Consider the {\em polar coordinates} $(\tau,\theta)(x):= (f(x), x/f(x))$, for $ x\in \rr^d\setminus\{f=0\}$.
By Fact \ref{fact:disintegration}, the measure $\nu$ satisfies the disintegration formula 
\eqref{e:fact:disintegration}, with spectral measure
$$
  \sigma_S(B):= \nu ((f,\theta) \in (1,\infty)\times B),
$$
on the (finite) unit sphere $S = \{ f = 1\}\cap \rr^d$.
 
 \begin{theorem}\label{thm:o-statistics} Let $f:\overline \rr^d\to[0,\infty]$ be a continuous homogeneous loss function, 
 such that $f(x) <\infty$ for all $x\in \rr^d$.  Suppose that $X\in RV_{\alpha}(a_n,\{f=0\},\nu)$ and 
 $X_i,\ i=1,\ldots,n$ are independent copies of $X$.\\
 
{\it (i)}  Consider the Point process ${\cal N}_n:= \{ a_n^{-1} X_i,\ i=1,\ldots,n\}$. Then, as $n\to\infty$
 \begin{equation}\label{e:thm:o-statistics-i}
  {\cal N}_n \cap \{f>0\} \Longrightarrow {\cal N},
 \end{equation}
 where ${\cal N}$ is a Poisson process on $\rr^d\setminus\{ f= 0\}$ with intensity $\nu$ and 
 `$\Rightarrow$' denotes weak convergence of probability distributions on the space of random 
 point measures equipped with the vague convergence topology.\\
 
{\it (ii)} Moreover, with $c:= \nu\{f>1\} = \sigma_S(S)$, we have
 \begin{equation}\label{e:thm:o-statistics-ii}
  {\cal N} \stackrel{d}{=} c^{1/\alpha} {\Big\{} \Gamma_k^{-1/\alpha} \Theta_k,\ k\in\N {\Big\}},
 \end{equation}
 where $1<\Gamma_1<\Gamma_2<\cdots$ is a standard Poisson process on $(0,\infty)$. 
 The $\Theta_k$'s are iid and independent of the $\Gamma_k$'s random variables taking values on 
the unit sphere $S$ and having distribution $\sigma_S(\cdot)/c.$\\

{\it (iii)} In particular, for all $m\in\N$, as $n\to\infty$,
\begin{equation}\label{e:thm:o-statistics-iii}
  \frac{1}{a_n} ( X_{k(i;n)},\ i=1,\ldots,m ) \Longrightarrow c^{1/\alpha} (\Gamma_i^{-1/\alpha} \Theta_i,\ i=1,\ldots,m).
\end{equation}
 \end{theorem}
  
  \begin{proof} By Theorem 5.3 (i) on p.\ 138 in \cite{resnick:2007}, the fact that
  $X \in RV_\alpha(\{a_n\},D:=\{f=0\},\nu)$ is equivalent to \eqref{e:thm:o-statistics-i}.  This completes the proof
  of part {\it (i)}.
  
  Part {\it (ii)} follows readily from the disintegration formula \eqref{e:fact:disintegration}.  Indeed, let ${\cal N}$ denote
  the Poisson process on the right--hand side of \eqref{e:thm:o-statistics-ii} and let $\widetilde \nu$ be its 
  intensity. To prove \eqref{e:thm:o-statistics-ii}, it is enough to show that $\nu=\widetilde \nu$. 
    
    Let $T(x):= (f(x),\theta(x))$ and consider the rectangle sets $A_{r,B} =T^{-1}((r,\infty)\times B)$ for
   $r>0$ and measurable $B\subset S$. Since the class of such rectangle sets forms a $\pi$-system that
   generates the $\sigma$-algebra on $\rr^d\setminus D$, it is enough to show that 
   $\nu(A_{r,B})=\widetilde \nu(A_{r,B})$, for all $r>0$, and measurable $B\subset S$.
   
  Since $f$ is $1$-homogeneous and $f(\Theta_i) = 1$, we have that
  $$
   T(c^{1/\alpha} \Gamma_i^{-1/\alpha}\Theta_i) = \Big(c^{1/\alpha} \Gamma_i^{-1/\alpha},\Theta_i\Big).
  $$
  Therefore, for all $r>0$ and measurable $B\subset S$, for $A_{r,B} = T^{-1}((r,\infty)\times B)$, we have
  \begin{eqnarray}\label{e:thm:o-statistics-1}
   P ( \widetilde {\cal N} \cap A_{r,B} = \emptyset ) &=& 
   P\Big( \Big\{ i\in\N\, :\, c^{1/\alpha}\Gamma_i^{-1/\alpha} \in (r,\infty),\ 
   \Theta_i\in B\Big\} = \emptyset\Big)\nonumber\\ 
   &= &
    \sum_{n=0}^\infty P(\Theta_1\in B^c)^n 
    P(|\Pi \cap ( r,\infty)| = n),
    \end{eqnarray}
    where $\Pi$ denotes the Poisson process $\{c^{1/\alpha} \Gamma_i^{-1/\alpha},\ i\in \N\}$. 
    The latter equals
  \begin{eqnarray}\label{e:thm:o-statistics-2}
  \sum_{n=0}^\infty P(\Theta_1\in B^c)^n  \frac{(cr^{-\alpha})^n}{n!} e^{-c r^{-\alpha}}
  &=&e^{-cr^{-\alpha}} e^{cr^{-\alpha} P(\Theta_1\in B^c)} \nonumber\\
   e^{ - cr^{-\alpha} P(\Theta_1\in B)} &=& 
   \exp{\Big\{ } - c\int_r^\infty \frac{\alpha d\tau}{\tau^{\alpha+1}} \frac{1}{c}\sigma_S(B) {\Big\}}\nonumber\\
   & =& \exp{\Big\{ } -  \int_0^\infty\int_S 1_{A_{r,B}}(\tau \theta) \frac{\alpha d\tau}{\tau^{\alpha+1}}  
   \sigma_S(d\theta) {\Big\}}.
   \end{eqnarray}
   Since $P( \widetilde {\cal N} \cap A_{r,B} = \emptyset ) = \exp\{ -\widetilde\nu(A_{r,B})\}$, Relations 
   \eqref{e:thm:o-statistics-1} and \eqref{e:thm:o-statistics-2} imply that
   $$
   \widetilde \nu(A_{r,B}) = \int_0^\infty\int_S 1_{A_{r,B}}(\tau \theta) \frac{\alpha d\tau}{\tau^{\alpha+1}}  
   \sigma_S(d\theta).
   $$
   This, in view of the disintegration formula \eqref{e:fact:disintegration}, yields $\nu(A_{r,B})
   =\widetilde \nu(A_{r,B})$ and hence $\nu=\widetilde \nu$.
   
   We now prove part {\it (iii)}. Observe that the map $T \equiv (f,\theta) :\rr^d\setminus\{f=0\} \to (0,\infty)\times S$ is
   a homeomorphism.  Therefore, we can equivalently view the convergence in \eqref{e:thm:o-statistics-i} in 
   polar coordinates.  More precisely, by letting $F_{n,i} := f(a_n^{-1} X_i)$ and $\Theta_{n,i}:= \theta(a_n^{-1}X_i)$,
   the continuous mapping theorem applied to \eqref{e:thm:o-statistics-i}, yields
   \begin{equation}\label{e:thm:o-statistics-3}
    T({\cal N}_n\cap \{f>0\}) \equiv  \{ (F_{n,i},\Theta_{n,i}),\ i = 1,\ldots,n \} \cap (0,\infty)\times S \Longrightarrow
    \Big\{ \Big(c^{1/\alpha} \Gamma_i^{-1/\alpha}, \Theta_i\Big),\ i\in\N\Big\},
   \end{equation}
   as $n\to\infty$.  Note that $T(c^{1/\alpha}\Gamma_i^{-1/\alpha}\Theta_i) =  
   (c^{1/\alpha} \Gamma_i^{-1/\alpha}, \Theta_i)$.
   
   Now, given a point measure $\Pi_n:=\{(f_i,\theta_i),\ i=1,\ldots,n\}$ in $\rr^d\setminus\{f=0\}$,
   introduce the {\em order statistics}  map:
   $$
    G_m(\Pi_n):= \Big((f_{k(1;n)},\theta_{k(1;n)}),\cdots, (f_{k(m;n)},\theta_{k(m;n)}) \Big),
   $$
   where $f_{k(1;n)} \ge f_{k(2;n)} \ge \cdots \ge f_{k(m;n)}$ are the top order statistics of
   the sample $f_i,\ i=1,\ldots,n$ with ties resolved by taking the indices in an increasing order.
   If $n<m$, we formally let $G_m(\Pi_n) = ((1,\theta_0),\cdots,(1,\theta_0))$ for some fixed $\theta_0\in S$.
   
   Let $M_p(\rr^d\setminus\{f=0\})$ denote the space of locally finite point measures equipped with the
   vague convergence topology.  It is easy to show that the so-defined map 
   $G_m: M_p(\rr^d\setminus\{f=0\}) \to \Big( (0,\infty)\times S \Big)^m$ 
   is continuous on the range of the Poisson point process 
   $T({\cal N}) = \{(c^{1/\alpha}\Gamma_i^{-1/\alpha},\Theta_i),\ i\in\N\}$. This is because there are no ties
   among the $\Gamma_i$'s (with probability one) and moreover
   $$
   G_m({\cal N}) = \{ (c^{1/\alpha} \Gamma_i^{-1/\alpha}, \Theta_i),\ i=1,\ldots,m\}.
   $$
   The continuous mapping theorem applied to \eqref{e:thm:o-statistics-3} then implies 
   $G_m({\cal N}_n) \Rightarrow G_m({\cal N})$, as $n\to\infty$.  Since, as $n\to\infty$, with probability 
   converging to one, at least $m$ of the losses $f(X_i),\ i=1,\ldots,n$ are positive, we have
   $$
   P \Big( G_m ({\cal N}_n)  = \Big\{ \Big(f(a_{n}^{-1} X_{k(i;n)}), \theta( X_{k(i;n)})\Big),\ i=1,\ldots,m\Big\} \Big) \longrightarrow 
   1,\ \ \mbox{ as }n\to\infty.
   $$
   This implies
   \begin{eqnarray}
  & &    \Big\{ \Big(f(a_{n}^{-1} X_{k(i;n)}), \theta( X_{k(i;n)})\Big),\ i=1,\ldots,m\Big\} \nonumber\\
  & &  \quad \quad
   \Longrightarrow \{ (c^{1/\alpha} \Gamma_i^{-1/\alpha}, \Theta_i),\ i=1,\ldots,m\},\label{e:thm:o-statistics-4}
   \end{eqnarray}
   as $n\to\infty$, where the last convergence is in the sense of weak convergence of probability 
   distributions on $((0,\infty)\times S)^m$.  Another application of the continuous mapping theorem to 
   \eqref{e:thm:o-statistics-4} with the map $T^{-1}$ applied component-wise yields 
   \eqref{e:thm:o-statistics-iii} and the proof is complete.
   \end{proof}
 
\section{Examples}\label{sec:examples}

Let $f:\rr^d \to [0,\infty)$ be a continuous $1$-homogeneous function.  Suppose also that $f$ {\em extends}
to a continuous function $f:\overline \rr^d \to [0,\infty]$. This is a non--trivial requirement as shown in Remark 
 \ref{rem:continuity-at-infinity}. Letting $D:= \{f=0\}$, we then obtain that $\tau:=f$ and $\theta_f(x):=x/f(x)$ can 
 serve as {\em polar coordinates} in $\overline\rr^d\setminus D$. This, since $f(\theta) = 1$,
 simplifies the stochastic representation in Proposition \ref{p:thm1-structure} to
 $$
  Y \stackrel{d}{=} C^{1/\alpha} Z \Theta,\ \ \mbox{ with }C=\sigma_S(S),
 $$  
 where $\Theta \sim \sigma_S(\cdot)/\sigma_S(S)$. In particular, if 
 $C = \nu\{f>1\} = \sigma_S(S) =1,$
 we obtain 
 $$
  Y \stackrel{d}{=} Z \Theta.
 $$
 These laws will be referred to as {\em standard} $f$-implicit max-stable. They are obtained by simply rescaling an $f$-implicit max-stable vector with the constant $\nu\{f>1\}^{1/\alpha}$.

 \begin{example}[Pareto--Dirichlet implicit max--stable laws] \label{ex:Pareto-Dirichlet-continued} Consider Example \ref{ex:Pareto-Dirichlet} where
 $X = (1/U_i^{1/\alpha_i})_{i=1}^d$ is a vector of independent standard $\alpha_i-$Pareto components. As shown therein,
 we have $X \in RV_\alpha(\{n^{1/\alpha}\},D,\nu)$, where $\alpha= \sum_{i=1}^d$.  Let, as in that example, 
 $$
  f(x_1,\cdots,x_d) = {\Big(}\sum_{i=1}^d \frac{1}{x_i} {\Big)}^{-1}.
 $$
 Proposition \ref{p:thm1-structure} and the representation of the spectral measure imply that the {\em standard} $(f,\nu)$-implicit max-stable vector 
 $W$ has the following stochastic representation:
 \begin{equation}\label{e:PD-implicit-max-stable}
  W = Z \Theta \stackrel{d}{=} {Z/\xi}, 
 \end{equation}
 where $\xi = (\xi_1,\cdots,\xi_d)$ has the Dirichlet$(\alpha_1,\cdots,\alpha_d)$ distribution.  We shall refer to $W$ in \eqref{e:PD-implicit-max-stable} 
 as to a {\em Pareto--Dirichlet implicit max--stable} distribution.
 
 This discussion  suggests that that any other $(f,\nu)$-implicit max-stable law, with different homogeneous 
 function $f$ can be represented by {\em tilting} $W$ in \eqref{e:PD-implicit-max-stable}.
 
 \begin{prop} \label{p:Pareto-Dirichlet-tilting} Let $\nu$ be as in \eqref{e:Pareto-alpha-nu} 
 and let $Y$ be $(f,\nu)$-implicit max--stable. Then, for all bounded measaurable function $h$, we have
 \begin{equation}\label{e:p:Pareto-Dirichlet-tilting}
 \E h(Y) = c^{-\alpha} \E {\Big[} h( cZ  \Theta/f(\Theta) ) f^\alpha(\Theta) {\Big]},  
 \end{equation}
 where $c^\alpha =  \E f^{\alpha}(\Theta)$,  and where $Z$ and $\xi:=1/\Theta$ are independent standard $\alpha$-Fr\'eceht and Dirichlet$(\alpha_1,\cdots,\alpha_d)$,
 respectively.
 \end{prop}
 
 The proof is an immediate consequence of Relation \eqref{e:PD-implicit-max-stable} and Proposition \ref{p:thm1-structure}. This result shows a type of change 
 of measure representation for general implicit max--stable laws that are regularly varying with exponent measure $\nu$. Unlike the classical case, where the
  spectral measure of a multivariate (explicit) max--stable law completely determines the distribution up to a scaling factor. The implicit max--stable laws depend 
  in a non--trivial way on the underlying function $f$.
 
 \begin{remark}
 Proposition \ref{p:Pareto-Dirichlet-tilting} can be used to efficiently simulate functionals of
 $(f,\nu)$-implicit max-stable distributions using {\em importance sampling}. 
 \end{remark}
 \end{example}
 
 
 \begin{example}[Classic regular variation] Let $ X = (X^{(i)})_{i=1}^d \in RV_\alpha(\{a_n\},\{0\},\nu)$,
 i.e., we have regularly variation in the usual cone $\rr^d\setminus \{0\}$.  In this case, we also have
 $X \in RV_\alpha(\{a_n\}, D, \nu\vert_{\rr^d_D})$, for any cone $\rr^d_D$ such that $\nu(\rr^d_D)>0$.
 Thus, for any continuous homogeneous loss function $f$ such that $\nu(\{f>0\})>0$, 
 by Theorem \ref{thm1}, the implicit maxima of independent copies of $X$ converge to 
 a non-trivial $(f,\nu)$-implicit max-stable law.  The structure of these distributions depends on the loss and
 the spectral measure $\sigma$ of $\nu$. They can be readily expressed as shown in Proposition 
 \ref{p:thm1-structure}. Ultimately, a variety of implicit max--stable models, tailored to specific losses and
 applications can be developed. This is beyond the scope of the present work.
 
 In this example, we discuss the case of elliptical losses.  Since $X \in RV_\alpha(\{a_n\},\{0\},\nu)$,
 in this case, {\em any} norm $\rr^d$ leads to valid polar coordinates. Let for example $\tau(x):= \|x\|_2$ 
 be the Euclidean norm. Then, by standard $(\|\cdot\|_2,\nu)$-implicit max--stable law has the 
 following representation
 $$
  W = Z \Theta,
 $$
 where $\Theta$ has distribution $\sigma_0(\cdot) := \nu(\theta^{-1}(\cdot) \cap \{\tau>1\})/\nu(\{\tau>1\})$ on the unit
 Euclidean sphere $S$.
 
 By complete analogy with Proposition \ref{p:Pareto-Dirichlet-tilting}, Relation \eqref{e:p:Pareto-Dirichlet-tilting} holds,
 for any $(f,\nu)$-implicit max-stable vector $Y$. This allows us to simulate $Y$ through tilting. For example, one
 can determine the structure of all such laws where $f$ has {\em elliptical} contours. That is, suppose 
 $$
  f (x) = \psi( x^\top \Sigma x) =: \psi (\| x\|_{\Sigma}^2 ) ,
 $$
 where $\Sigma$ is a symmetric positive definite matrix and $\psi:[0,\infty)\to \rr$ is strictly monotone. 
 By Remark \ref{rem:composition}, Theorem \ref{thm1} applies to the continuous and $1-$homogeneous 
 function $(\psi^{-1}\circ f)^{1/2}(x) =  \|x\|_{\Sigma}^{1/2}$. Therefore the implicit extreme value laws $Y$ 
 corresponding to $f$ (equivalently, the $(f,\nu)$-implicit max-stable ones) have the representation
 $$
  \E  h(Y) = c^{-\alpha} \E {\Big[} h( cZ \Theta/\|\Theta\|_{\Sigma}) \|\Theta\|_{\Sigma}^{\alpha} {\Big]},
 $$
 where $c^{\alpha} = \E \|\Theta\|_\Sigma^\alpha$ and where $\Theta\sim \sigma_0$.

 Note that the distribution of $Y$, as expected, does not depend on $\psi$. Suppose for example 
 that $1/\psi(x) \, \propto\  \phi_{\Sigma}(x)$ is the density of centered multivariate Normal distribution 
 with covariance matrix $2\Sigma^{-1}$.  Suppose also that $X_i$ are as in Theorem \ref{thm1} and let 
  $$
   k(n):=  \mathop{\rm Argmax}_{i=1,\cdots,n} f(X_i) = \mathop{\rm Argmin}_{i=1,\cdots,n} \phi_{\Sigma}(x).
  $$
  In this case, the limit distribution of $a_n^{-1} X_{k(n)}$ describes the large--sample behavior of {\em novelties} relative
  to the Gaussian model $\phi_{\Sigma}(x)$ in the sense of Clifton\, {\em et al} \cite{clifton:hugueny:tarassenko:2011}. 
  
 \end{example}

\begin{example}[Gaussian copula] \label{ex:Gaussian} Let $Z = (Z_1, Z_2)^\top$ be bivariate Normal random vector
having standard Normal margins and correlation $\rho = \E (Z_1 Z_2) \in (-1,1)$.  
Let $\overline \Phi(z) = P(Z_1 >z),\ z\in \rr$ denote the complementary cdf of $Z_1$.  Consider the random 
vector
$$
 X = \Big(\frac{1}{\overline \Phi(Z_1)}, \frac{1}{\overline\Phi(Z_2)} \Big)^\top.
$$
Observe that $X$ has standard unit Pareto marginals and its dependence is determined by the 
Gaussian copula.  It is well known that the components of $X$ are asymptotically independent, or equivalently
that $X \in RV_1(\{0\},\nu)$, where the measure $\nu$ concentrates on the two positive axes. If one excises the axes and
considers regular variation in $(0,\infty)^2$, however, a finer {\em hidden regular variation} emerges.  
More precisely, letting $\rr^2\setminus D = (0,\infty)^2$, by Example 2.1 (p.\ 255) in Draisma {\em et al}
\cite{draisma:drees:ferreira:dehaan:2004}, we have that 
$$
 X \in RV_\alpha ( \{a_n\}, D, \nu),\ \ \mbox{ where } \alpha = \frac{2}{1+\rho}
 $$
 and for all $(x_1,x_2)\in (0,\infty)^2$, 
$$
 \nu((x_1,\infty)\times(x_2,\infty)) = (x_1x_2)^{-1/(1+\rho)}. 
$$
This shows that the random vector $X$ has the same regular variation behavior as in Example 
\ref{ex:Pareto-Dirichlet-continued} with $d=2$ and $\alpha_1=\alpha_2 = 1/(1+\rho)$. Therefore, 
with a homogeneous function $f(x_1,x_2) = (1/x_1 + 1/x_2)^{-1}$, for example, 
the $(f,\nu)$-implicit max-stable distribution attracting $X$ is of the form \eqref{e:PD-implicit-max-stable}. 
All results for Pareto--Dirichlet laws above apply in this particular setting.

\begin{remark} The derivation of the regular variation behavior on $(0,\infty)^d$ with $d=2$ 
for Gaussian copula with Pareto margins is rather technical.  To the best of our knowledge, the 
$d$-dimensional case $d>3$ remains open.
\end{remark}

\begin{remark}
The study of finer behavior of asymptotically independent variables was initiated with the seminal work of 
Ledford and Tawn \cite{ledford:tawn:1996} (see also 
\cite{heffernan:2000, heffernan:resnick:2007,fougeres:soulier:2010,dehaan:zhou:2011} 
among others).
\end{remark}

\end{example}
\section{Acknowledgements} The authors are grateful for inspiring discussions with Mark Meerschaert, 
Laurens de Haan, Cl\'ement Dombry and Gennady Samorodnitsky.  Section \ref{sec:o-stat} was motivated 
by a conversation with Gennady Samorodnitsky and Cl\'ement Dombry.  Example \ref{ex:Gaussian} 
was finished with the help of Laurens de Haan.

\appendix
\section{Some proofs and auxiliary Lemmas} \label{sec:appendix}

\begin{proof}[Proof of Fact \ref{fact:compact}]
 All open sets in  $\overline\rr^d_D$ are precisely of the type $V\setminus D$, where $V$ is open in $\overline \rr^d$.

$(\Leftarrow)$ Let $U_n := V_n \setminus D,\ n\in \N$ be an open cover of $F$ in $\overline \rr^d_D$, where the $V_n$-s are open
in $\overline \rr^d$. Observe that since $F$ is closed in $\overline \rr^d_D$, then $F = K\setminus D$, for some $K$ that is closed in 
$\overline\rr^d$. Since $D$ is also closed, $F\cup D = K\cup D$ is closed and hence compact in $\overline\rr^d$. Now, the fact that the open set $U$ covers
$D$, implies that $\{U,\ V_n,\ n\in \N\}$ is an open cover of the compact $K\cup D$ in $\overline \rr^d$. Thus, there exists a finite $N$, such that
$$
 F\cup D\equiv K\cup D  \subset U \cup \bigcup_{n=1}^N  V_n.
$$ 
This, since $F\cap U = \emptyset $ implies that $F\subset \cup_{n=1}^N V_n\setminus D \equiv \cup_{n=1}^N U_n$, which is 
a finite sub-cover of $F$ in $\overline \rr^d_D$, showing that $F$ is compact in $\overline \rr^d_D$. 

\smallskip
$(\Rightarrow)$ For all $\epsilon>0$, let $D_\epsilon := \{ x\in \overline \rr^d\, :\, \rho(x,D)\le \epsilon\}$, where 
$\rho(x,D)= \min_{y\in D} \rho(x,y)$ 
is the distance from $x$ to the compact $D$ in $\overline \rr^d$ with $\rho$ as in \eqref{e:rho}. Note that the sets $D_\epsilon\ (\epsilon>0)$ are closed and 
$D = \cap_{\epsilon>0} D_\epsilon$. Let $U_\epsilon:= \overline \rr^d\setminus D_\epsilon$ and observe that $\{U_\epsilon,\ \epsilon>0\}$ is
an open cover of $F$ in $\overline \rr^d_D$. Since $F$ is compact, it is also covered by a finite subset of $U_\epsilon$-s. Since the latter are nested, it 
follows that $F\subset U_{\epsilon_0}$, for some $\epsilon_0$. By taking $U:= \{ x\in \overline \rr^d\, :\, \rho(x,D)< \epsilon_0\}$, we obtain that  
$F \cap U$ and $D\subset U$, which shows that $F$ is bounded away from $D$.
\end{proof}

\begin{proof}[Proof of Proposition \ref{p:RV-polar}]

($\Rightarrow$) The continuity and homogeneity of $\tau$ imply 
that $\partial \{ \tau>t\}  = \{\tau = t\} = t\{\tau =1\},\ t>0$. Since $\nu\{\tau>\epsilon\} <\infty$ for any
$\epsilon>0$ and since the set $\{\tau>\epsilon\}$ equals the disjoint union 
$\cup_{t>\epsilon} \{\tau = t\} = \cup_{t>\epsilon} t\{\tau>1\} $, we obtain that $\nu(\{\tau =t\})=0,\ \forall t>0$, 
i.e.,  $\{\tau>t\}$ is a $\nu$-continuity set for all $t>0$. Thus, in view of the homogeneity of $\tau$ and the scaling
property of $\nu$, Relation \eqref{e:d:RV-D} implies that, for all $x>0$, as $n\to\infty$,
\begin{equation}\label{e:p:RV-polar-1}
  n P(a_n^{-1} \tau (X) >x ) = n P( a_n^{-1} X \in \{ \tau>x\} ) \longrightarrow \nu( \{\tau>x\}) \equiv \nu(\{\tau>1\}) x^{-\alpha}.
\end{equation}
We have moreover that the function $u\mapsto P(\tau(X)>u)$ varies regularly, 
with exponent $(-\alpha)$.  This follows from its 
monotonicity and Theorem 1.10.3 on p.\ 55 in \cite{bingham:goldie:teugels:1987}.

Introduce now the probability measures $Q_u(\cdot) := P(\theta(X) \in \cdot | \tau(X)>u),\ u>0$ on 
$(\overline S,{\cal B}(\overline S))$, where the measure is extended as zero on the infinite points in 
$\overline S \setminus S$, with $\overline S:= \{\tau = 1\}$. We will first show that $Q_{a_n} \stackrel{w}{\to} \sigma_0$, 
for some probability measure $\sigma_0$. Indeed, by \eqref{e:p:RV-polar-1}, as $n\to\infty$
\begin{eqnarray*}
 Q_{a_n} (B) &=& \frac{P( \theta(X) \in B, \tau(X)> a_n)}{P(\tau(X)>a_n)} \\
 &\sim& \frac{n P( a_n^{-1} X \in {T}^{-1} ( (1,\infty]\times B)  )}{\nu\{\tau>1\}} :=\mu_n({T}^{-1} ( (1,\infty]\times B)),
\end{eqnarray*}
where ${T}(x):=(\tau(x),\theta(x))$. Recall that ${T} : \overline \rr^d_D \to (0,\infty) \times \overline S$ is a 
homeomorphism (recall \eqref{e:T-def} and the discussion thereafter). Therefore,
for every open set $B \subset \overline S$ (in the relative topology), the set $A:= {T}^{-1}( (1,\infty)\times B) $ is open. Further, since $A\subset\{\tau\ge 1\}$, where
$\{\tau\ge 1\}$ is compact (in $\overline \rr^d_D$), 
the vague convergence in \eqref{e:d:RV-D} coincides with the weak convergence of finite measures 
restricted to the set $\{\tau\ge 1\}$.  Hence, by the Portmanteau characterization of weak convergence (see e.g.\ Theorem A 2.3.II on p.\ 391 in 
\cite{daley:vere-jones:2003}), Relation \eqref{e:d:RV-D} implies that, for all open sets $B \subset \overline S$, 
$$
 \liminf_{n\to\infty} Q_{a_n}(B) = \liminf_{n\to\infty} \mu_n( A ) \ge \frac{\nu(A)}{\nu\{\tau>1\}}  =:\sigma_0(B),
$$
Since the last relation is valid for all open sets $B \subset \overline S$, the Portmanteu theorem applied to the probability measures $Q_{a_n}$,
shows that $Q_{a_n}\stackrel{w}{\to} \sigma_0,$ as $n\to\infty$.

Let now $u_n\to\infty$ be arbitrary. We will show that $Q_{u_n}\stackrel{w}{\to} \sigma_0$. Let $a_n^*:= \inf_{m\ge n} a_m$ and
note that by Theorem 1.5.3 on p.\ 23 of \cite{bingham:goldie:teugels:1987}, we have $a_n^* \sim a_n,\ n\to\infty$. Further, the
fact that $u\mapsto P(\tau(X)>u)$ is regularly varying at infinity, implies that $P(\tau(X)>a_n^*) \sim P(\tau(X)> a_n)$ since the convergence
$nP(\tau(X)>a_n x)\to C x^{-\alpha}$ is uniform in $x$ on each fixed interval $[c,\infty)$, $c>0$ (cf Theorem 1.5.2 in \cite{bingham:goldie:teugels:1987}).
Since $a_n^*\uparrow \infty$, there exists an integer sequence $k_n\to\infty$, such that for all sufficiently large $n$,
$$
 a_{k_n}^* \le u_n < a_{k_{n}+1}^*.
$$
Hence, for any measurable $B\subset S$, we have
\begin{equation}\label{e:p:RV-polar-2}
 \frac{P(\theta(X) \in B, \tau (X) \ge a_{k_n+1}^* )}{ P(\tau(X)> a_{k_n}^*)}  < Q_{u_n}(B) \le  \frac{P(\theta(X) \in B, \tau (X) \ge a_{k_n}^* )}{ P(\tau(X)> a_{k_n+1}^*)}.
\end{equation}
By the fact that $m P(\tau(X)>a_m^*)\to C>0$, we get $P(\tau(X)> a_{k_n+1}^*) \sim  P(\tau(X)> a_{k_n}^*)$. This, since $Q_{a_n^*}\stackrel{w}{\to}\sigma_0$,
shows that the upper and lower bounds of $Q_{u_n}(B)$ in \eqref{e:p:RV-polar-2} converge to $\sigma_0(B)$, for all continuity sets $B$, which 
completes the proof of the `only if' part.  The fact that $\nu$ does not put charge on the infinite points implies
that $\sigma_0$ concentrates on $S\equiv \overline S\cap \rr^d$.

\medskip
($\Leftarrow$) Suppose now that \eqref{e:p:RV-polar-1} holds. Consider the semiring of subsets of $\overline \rr_D^d$,
$$
 {\cal R} := \{ {T}^{-1}((x,y] \times B)\, :\, 0<x<y< \infty,\ B \in {\cal B}(S)\},
$$
where ${T}(x) = (\tau,\theta)$ and ${\cal B}(S)$ is the class of Borel measurable sets in $S$. Define 
the $\sigma$-finite measure 
$$
 \nu( {T}^{-1}( (x,\infty)\times B) ) := C x^{-\alpha}\sigma_0(B),\ \ (x,\infty]\times B \in {\cal R}. 
$$
Since ${\cal R}$ is a $\pi$-system generating the Borel $\sigma$-algebra ${\cal B}(\overline \rr_D^d)$, the mapping
$\nu$ uniquely entends to a $\sigma$-finite measure on $(\rr_D^d,{\cal B}(\rr_D^d))$. We further extend $\nu$ 
to $(\overline \rr_D^d, {\cal B}(\rr_D^d))$ by defining as zero at infinity.

By \eqref{e:p-RV-polar}, we readily obtain that for all $A = {T}^{-1}((x,\infty)\times B)$, 
where $x>0$ and where $B\in {\cal B}(S)$ is a continuity set of $\sigma_0$, that 
$$
 nP( a_n^{-1} X \in A ) = P( \theta(X) \in B | a_n^{-1} \tau(X)>x) P(a_n^{-1} \tau(X)>x) \longrightarrow \nu( A),
$$
as $n\to\infty$. Thus, we also have that $n P(a_n^{-1} X \in A) \to \nu(A),\ n\to\infty$, for all sets in the semi ring 
${\cal R}$ such that $\theta(A)$ is a $\sigma_0$-continuity set.

To prove \eqref{e:d:RV-D}, since $\nu$ is supported on $\rr_D^d$, it is enough to show weak convergence of
the measures restricted to $\{\tau\in(\epsilon,\infty)\}$, for each $\epsilon>0$. Note however that the restriction of 
${\cal R}$ to $\{\tau\in (\epsilon,\infty)\}$ is a covering semi ring for 
the separable metric space $\{\tau\in (\epsilon,\infty)\}$. Therefore, ${\cal R}$ is a convergence determining class 
(cf Proposition A 2.3.IV on p.\ 393 in \cite{daley:vere-jones:2003}) and the already established weak convergence for 
${\cal R}$ implies the result.
\end{proof}

\medskip
The following result was used in the proof of Theorem \ref{thm1}. Consider a locally compact metric space $(E,\rho)$ with countable base equipped 
with its Borel $\sigma$-algebra. More precisely, we shall assume that all closed and $\rho$-bounded sets in $E$ are compact. Recall that a set $A\subset E$ is
$\rho$-bounded, if $A$ is contained in a ball $B(x,r) = \{y\in E\, :\, \rho(x,y) < r\}$, for some $x\in E$ and $r>0$. Borel measures that
are finite on all compacts are referred to as Radon.

\begin{lemma}\label{l:weak-convergence} Let $\nu_n$ and $\nu$ be Radon measures on $E$. Suppose that $h_n$ and $h$ are locally bounded
and non--negative measurable functions defined on $E$. Introduce the Radon measures
$$
 \mu_n(A):= \int_A h_n(x) \nu_n(dx)\ \ \mbox{ and }\ \ \mu(A):= \int_A h(x) \nu(dx),\ A\in {\mathcal B}(E).
$$
Suppose that $\nu(\overline{{\rm Disc}(h)}) = 0$ and that for every compact $K\subset E$, we have that
$\sup_{x\in K} |h_n(x) -h(x)| \to 0,\ n\to\infty,$
that is, $h_n$ converges to $h$ locally uniformly. 

Then, the convergence $\nu_{n}\stackrel{v}{\to} \nu,$  as $n\to\infty$ implies $\mu_n \stackrel{v}{\to} \mu$, as $n\to\infty.$
Furthermore, if $\mu_n$ are probability measures, then $\mu(E)\le 1$.
\end{lemma}
\begin{proof} We need to show that for all continuous functions $g$ with compact support, we have
$\int_E g d\mu_n \to \int_E g d\mu,$ as $n\to\infty,$ or equivalently,
$\int_E gh_n d\nu_n \to \int_E gh d\nu,$ as $n\to\infty.$
By the triangle inequality, we have that
\begin{equation}\label{e:l:weak-convergence-0}
 {\Big|} \int_E  gh_n d\nu_n - \int_E gh d\nu {\Big|}\le \int |g(h_n - h)| d \nu_n + {\Big|}\int g h d\nu_n - \int g hd\nu{\Big|} =: I_n +J_n.
\end{equation}
For the term $I_n$, for all $\delta>0$, we have
\begin{equation}\label{e:l:weak-convergence}
 I_n \le C \sup_{x\in K}|h_n(x) -h(x)| \nu_n(K) \le C \sup_{x\in K}|h_n(x) -h(x)| \nu_n(K^\delta),
\end{equation}
where $C = \sup_{x\in E} |g(x)|$, the compact support of $g$ is denoted by $K$, and $K^\delta := \{ y\in E\, :\, \rho(x,y)\le \delta\}$ is its closed $\delta$-neighborhood. Note that
the closed and bounded set $K^\delta$ is compact and hence $\nu(K^\delta)<\infty$.  Also, $K^{\delta_1}$ is contained in the interior of $K^{\delta_2}$, for all
$0<\delta_1<\delta_2$ and hence $\partial ( K^\delta)$ are disjoint for all $\delta>0$. This implies that $\nu(\partial K^\delta) = 0$, for all but countably many $\delta>0$,
since the Radon measure $\nu$ is $\sigma$-finite and has at most countably many atoms. Thus, $K^\delta$ is a continuity set of $\nu$, 
for some $\delta>0$. This yields  $\lim_{n\to\infty} \nu_n(K^\delta) = \nu(K^\delta)<\infty,$ and the right-hand side of \eqref{e:l:weak-convergence} vanishes, 
as $n\to\infty$.

{\em Now, we focus on $J_n$ in \eqref{e:l:weak-convergence-0}.} Let $\delta>0$ and define
\begin{equation}\label{e:tau-delta}
 h_\delta(x):= h(x) \tau_\delta(x),\ \ \mbox{ where } \tau_\delta(x):= \exp {\Big\{}  \frac{1}{\delta} - \frac{1}{\delta \wedge\rho(x,F) }{\Big\}},
\end{equation}
with $F := \overline{ {\rm Disc}(h)} \cap K$. The function $\tau_\delta$ is continuous, $|\tau_\delta|\le 1$ and it vanishes on the set $F$. Further, 
$\tau_\delta(x) = 1,$ if $\rho (x,F)\ge \delta$.  By the triangle inequality, we have
$$
J_n\le \int_E |g(h - h_\delta)| d\nu_n  + {\Big|} \int_E g h_\delta (d\nu_n - d\nu) {\Big|} +  \int_E |g(h - h_\delta)| d\nu =: J_{n,1} + J_{n,2} + J_{n,3}.
$$
Note that for each fixed $\delta>0$, the function $gh_\delta$ is continuous and has compact support. Therefore, 
the vague convergence $\nu_n\stackrel{v}{\to}\nu,\ n\to\infty$ implies $J_{n,2}\to 0,\ n\to\infty$. 
Now, by the local boundedness of $h$ and the fact that $\{h\not=h_\delta\}\subset F^\delta,$ we have
$$
J_{n,1} + J_{n,3} \le C \nu_n( F^\delta )  + C \nu(F^\delta),
$$ 
where $C = \sup_{x\in K} |g(x) h(x)|<\infty$ and $F^\delta = \overline{\{ \tau_\delta < 1 \}}$. As argued above, for all
$0<\delta_1<\delta_2,$ the set $F^{\delta_1}$ is contained in the interior of $F^{\delta_2}$. Thus, 
the $\sigma$-finiteness of $\nu$ implies that the $F^{\delta}$s are $\nu$-continuity sets, for all but 
countably many $\delta>0$.  Further, we have $\nu(F^{\delta}) \downarrow \nu(F) = 0,$ as $\delta\downarrow 0$ 
since $\nu(F^\delta)<\infty$ and $F^{\delta} \downarrow F, \delta\downarrow 0$. Thus,  for 
every $\epsilon>0$, we can pick $\delta>0$ such that $F^\delta$ is a $\nu$-continuity set and $\nu(F^{\delta}) < \epsilon/(3C)$. This ensures that 
$J_{n,3} < \epsilon/3$ and $\nu_n (F^{\delta} ) \to \nu(F^{\delta}),\ n\to\infty$, so that $J_{n,1}< \epsilon/2$, for all sufficiently large $n$.
This, since $\epsilon>0$ was arbitrary yields $J_{n,1}+ J_{n,2} \to 0,\ n\to\infty$, which completes the proof.
\end{proof}


\begin{lemma}\label{l:f(X)-rv} Let $X$ and $f$ be as in Assumptions RV$_\alpha(D,\nu)$ and H. If in addition,  $f$ and $\nu$ satisfy 
Assumptions F and C, then

{\it (i)} The sets $A_u:=\{ f>u\} \equiv u \{f>1\},\ u>0$ are $\nu$-continuity sets for all but countably many $u$-s.

{\it (ii)} With $a_n$ as in \eqref{e:d:RV-D}, for all $y>0$, we have
\begin{equation}\label{e:l:f(X)-rv-ii}
n P( f(X)> a_n y) \longrightarrow \nu(\{f>1\}) y^{-\alpha},\ \ \mbox{ as }n\to\infty.
\end{equation}

{\it (iii)} The function $y\mapsto P(f(X) > y)$ is regularly varying of exponent $-\alpha$ and hence the convergence in 
\eqref{e:l:f(X)-rv-ii} is uniform in $y$ on $[c,\infty)$, for any fixed $c>0$.
\end{lemma}
\begin{proof}  {\it (i):} We need to show that $\nu(\partial A_u) = \nu(\overline A_u \setminus \langle A_u \rangle) = 0$, for all but countably many $u$-s.
If $f$ is continuous, then $\partial A_u= \{f =u\}$ and these sets are disjoint for all $u>0$. Hence, for all $\epsilon>0$,
$\nu(A_{\epsilon}) = \cup_{u>\epsilon} \partial A_u$, which is finite by Assumption F. This shows that $A_u$-s are $\nu$-continuity sets 
for all but countably many $u>0$. When $f$ is discontinuous, however, $\partial\{f>u\}\not = \{f=u\}$ and this argument fails. The intuition 
is that jumps in the angular component in \eqref{e:f-polar} can lead to non-trivial overlaps between the boundaries $\partial A_u$ 
for entire ranges of $u$-s. The role of Assumption  C is to make such overlaps negligible in $\nu$-measure. We shall now make this precise.

Fix an arbitrary $\epsilon>0$. By Assumption F, there exists an open set $U \subset \overline\rr^d$, such that $D\subset U$ and $\{f>\epsilon\} \subset U^c
:= \overline\rr^d\setminus U$. Since $U^c$ is closed, we also have that $\overline{\{f>\epsilon\}}\subset U^c$. Note that the two disjoint 
sets $D$ and $U^c$ are compact in $\overline \rr^d$ and hence they can be separated. Namely,
$$
 \rho(D,U^c) := \inf_{x\in D, y\in U^c} \rho(x,y) =: \epsilon_0 >0,
$$
where $\rho$ is the metric in \eqref{e:rho}.  Consider the open $\epsilon_0/2$-neighborhood of $D$: 
$$
  D_{\epsilon_0/2} = \{x\in\overline \rr^d\, :\, \rho(x,D)<\epsilon_0/2\}
$$
and define the following compact in $\overline \rr^d\setminus D$
\begin{equation}\label{e:set-F}
F:=\overline{{\rm Disc}(f)} \cap D_{\epsilon_0/2}^c.
\end{equation}
Observe that since $\{f>\epsilon\}\subset U^c\subset D_{\epsilon_0/2}^c$, we have
\begin{equation}\label{e:set-inclusion-star}
 {\rm Disc}(f) \cap \{f>\epsilon\} \subset \overline{{\rm Disc}(f)} \cap U^c \subset F \subset \overline{{\rm Disc}(f)}.
\end{equation}
In particular, $\nu(F) = 0$, because of Assumption C.

The intuition behind the set $F$ is that it collects all discontinuity points of $f$ over the region $\overline{\{f>\epsilon\}}$, and 
therefore over the regions $\overline{\{f>u\}}$, for $u>\epsilon$. We shall regularize $f$ and 
replace it by a function that is continuous on $D_{\epsilon_0/2}^c \supset \overline{\{f>\epsilon\}}$ and coincides with $f$,
except for a small neighborhood of $F$. Letting the neighborhood shrink, we will arrive at the desired claim. Now, the details.
 
For each $\delta \in (0,\epsilon_0/2)$, define the function $\tau_\delta$ as in \eqref{e:tau-delta}.  Note that 
$$
 F_\delta := \{ x \in \rr^d\, :\, \rho(x,F) <\delta\} = \{\tau_\delta < 1\}
 $$ 
 is the open $\delta$-neighborhood of $F$.  Since $F$ is bounded away from $0$, so are $F_\delta$ for all sufficiently small $\delta>0$. 
 Thus, $\nu(F_\delta)<\infty,$ eventually, as $\delta\downarrow 0$, and hence $\nu(F_\delta) \downarrow  \nu(\cap_{\delta>0} F_\delta) \equiv \nu(F) =0$.

We are now ready to study $\nu(\partial \{ f>u\})$ for $u>\epsilon.$  Define the functions
$f_\delta(x):= f(x) \tau_\delta(x),\ \delta \in (0,\epsilon_0/2)$. By construction, $f(x) = f_\delta(x)$ for all $x\in F_\delta^c$ and thus,
$$
\{f>u\} = {\Big(}\{f_\delta>u\}\cap F_\delta^c {\Big)} \cup {\Big(}\{ f> u \} \cap F_\delta{\Big)},
$$
which implies
$$
 \partial\{f>u\}  \subset \partial {\Big(}\{f_\delta>u\}\cap F_\delta^c{\Big)} \cup \partial {\Big(}\{ f> u \} \cap F_\delta{\Big)}.
$$
Now, using the facts that $\partial(A\cap B) \subset \partial A \cup \partial B$ and $\partial ( F_\delta^c ) = \partial F_\delta \subset \overline {F_\delta}$, we obtain
\begin{equation}\label{e:l:f(X)-rv-1}
\partial  \{f>u\}  \subset \partial  \{f_\delta >u\}  \cup \overline {F_\delta}.
\end{equation}
We will show next that $\nu(\partial \{f_\delta>u\})=0$, for all but countably many $u>\epsilon$.
Since $\tau_\delta\le 1$, we have $f_\delta \le f$ and thus
$$
 \overline{\{ f_\delta > u\}} \subset \overline{\{ f >\epsilon\}},\ \mbox{ for all $u> \epsilon $.}
$$
This, in view of \eqref{e:set-inclusion-star}, implies that $\overline {\{ f_\delta > u\}} \subset D_{\epsilon_0/2}^c$, for all $u>\epsilon$.
By the construction of the set $F$ in \eqref{e:set-F}, however, the function $f_\delta$ is continuous 
on $D_{\epsilon_0/2}^c$, and hence $\partial \{f_\delta >u \} = \{f_\delta = u\},$ for all $u>\epsilon$.  Thus, as argued above, 
the fact that $\nu(\{f_\delta>\epsilon\})<\infty$ implies $\nu( \partial \{f_\delta>u\}) = 0$ for all but countably many $u>\epsilon$.
(The set of $u$-s may depend on the choice of $\delta$.)  

On the other hand, $\overline{F_\delta} \subset F_{2\delta}$ and as shown $\nu(F_{2\delta})\downarrow 0$ as 
$\delta\downarrow 0$. Thus, by taking a limit over a countable sequence $\delta_m\downarrow 0$, we see that 
the $\nu$-measure of the left-hand side in \eqref{e:l:f(X)-rv-1} vanishes  for all but countably many $u>\epsilon$. 
This completes the proof of part {\it (i)}.

\smallskip
{\em We now prove {\it (ii)}.} For all $y>0$, by the homogeneity of $f$ (Assumption H),
\begin{eqnarray*}
n P(f(X)> a_n y) &=& n P(X \in f^{-1}(a_ny,\infty)) \\
 &=& n P(X \in a_n\{ f> y\} ) =: n P(X \in a_n A_y). 
\end{eqnarray*}
Now, by the already established part {\it (i)}, all but countably many $A_y$-s are $\nu$-continuity sets. Thus, by \eqref{e:d:RV-D}, 
\begin{equation}\label{e:l:f(X)-rv-2}
 n P(a_n^{-1} X\in  A_y ) \to \nu(A_y ) \equiv  \nu(\{ f>1 \})  y^{-\alpha},\ \ \mbox{ as }n\to\infty,
\end{equation}
for all but countably many $y$-s. The monotonicity (in $y$) of the left-hand side in \eqref{e:l:f(X)-rv-2} and the continuity (in $y$)
of the limit, imply that Relation \eqref{e:l:f(X)-rv-ii} holds for all $y>0$. This completes the proof of part {\it (ii)}. 

\smallskip
{\it (iii):} Observe that the sequence $a_n$ in \eqref{e:d:RV-D} is regularly varying with exponent $1/\alpha$, 
\eqref{e:l:f(X)-rv-ii} holds for all $y>0$, and the function $u\mapsto P(f(X)>u)$ is monotone. Therefore, Theorem 1.10.3 on p.\ 55 in \cite{bingham:goldie:teugels:1987} applies 
and shows that $u\mapsto P(f(X)>u)$ is regularly varying, with index $-\alpha$.  By Theorem 1.5.2 on p.\ 22 in \cite{bingham:goldie:teugels:1987}
the convergence in \eqref{e:l:f(X)-rv-ii} is also uniform in $y$ on $[c,\infty)$, for all $c>0$.
 \end{proof}
 
 \medskip
 The following slight reformulation of Scheffe's Lemma is useful.
 
 \begin{lemma}[induced Scheffe's Lemma] \label{l:Scheffe} Let $(E,{\cal E},\mu)$ be a measure space and let $T:(E,{\cal E}) \to (F,{\cal F})$ 
 be an ${\cal E}|{\cal F}$-measurable mapping. Suppose that $p_n, p \in L^1(E,{\cal E},\mu)$ are probability 
 densities and define the probability measures $Q_n$ and $Q$ on $(F,{\cal F})$ as follows:
 $$
  Q_n(B) := \int_{E} 1_B(T(x)) p_n(x)\mu(dx)\ \mbox{ and }\ \ Q(B) := \int_{E} 1_B(T(x)) p(x)\mu(dx),\ \ B\in {\cal F}.
 $$
 If $p_n(x)\to p(x),\ n\to\infty$, $\mu$-a.e., then 
 $$
 \| Q_n - Q\|_{\rm tv} := \sup_{B\in {\cal F}} |Q_n(B) - Q(B)| \longrightarrow 0, \ \ \mbox{ as }n\to\infty.
 $$
 \end{lemma}
 \begin{proof} Observe that that 
 $$
   |Q_n(B) - Q(B)| \le \int_{E} 1_{B}(T(x)) | p(x) - p_n(x) | \mu(dx) \le \|p_n - p\|_{L^1(\mu)}.
 $$
 The last bound vanishes by the classic version of Scheffe's lemma.
 \end{proof}

\bibliographystyle{plain}
\bibliography{../mst-bibfile}

\end{document}